\documentclass[a4paper,12pt]{article}
 \usepackage[english]{babel}
  \usepackage{amsmath,amsfonts,amssymb,amsthm}

  \usepackage{hyperref} 

 \usepackage[active]{srcltx} 
 \usepackage{color,soul} 

\newtheorem{theorem}{Theorem}[section]

\newtheorem{lemma}[theorem]{Lemma}
\newtheorem{proposition}[theorem]{Proposition}

\theoremstyle{definition}
\newtheorem{definition}[theorem]{Definition}
\newtheorem{remark}{Remark}

\def\N{\mathbb{N}}
\def\Z{\mathbb{Z}}

\def\R{\mathbb{R}}

\let\e=\varepsilon
\let\vp=\varphi
\let\t=\tilde
\let\ol=\overline
\let\ul=\underline
\let\.=\cdot
\let\0=\emptyset
\let\mc=\mathcal
\def\ex{\exists\;}

\def\dv{\text{\rm div}\,}

\def\Tr{\text{\rm Tr}}

\newcommand{\esssup}{\mathop{\rm ess{\,}sup}}
\newcommand{\essinf}{\mathop{\rm ess{\,}inf}}
\newcommand{\su}[2]{\genfrac{}{}{0pt}{}{#1}{#2}}

\def\eq#1{{\rm(\ref{eq:#1})}}
\def\thm#1{Theorem \ref{thm:#1}}
\def\seq#1{(#1_n)_{n\in\N}}
\def\limn{\lim_{n\to\infty}}

\def\pe{principal eigenvalue}

\newenvironment{formula}[1]{\begin{equation}\label{eq:#1}}
                       {\end{equation}\noindent}

\def\Fi#1{\begin{formula}{#1}}
\def\Ff{\end{formula}\noindent}

\setstcolor{red}

\def\lm#1{\left\lfloor{#1}\right\rfloor}
\def\um#1{\left\lceil{#1}\right\rceil}

\usepackage{xcolor}

\begin{document}
\title{Transition waves for Fisher-KPP equations with general time-heterogeneous 
and space-periodic coefficients}
\author{Gr\'egoire Nadin 
\footnote{CNRS, UMR 7598, Laboratoire Jacques-Louis Lions, F-75005, Paris, France}  
\footnote{Sorbonne Université\'es, UPMC Univ Paris 06, UMR 7598, Laboratoire Jacques-Louis Lions, F-75005, Paris,
France}
\and Luca Rossi 
\footnote{ Dipartimento di Matematica,
  Universit\`a di Padova, via Trieste 63, 35121 Padova, Italy}
}
\maketitle

\begin{abstract}
This paper is devoted to existence and non-existence results for generalized transition waves solutions of space-time heterogeneous Fisher-KPP equations. 
When the coefficients of the equation are periodic in space but otherwise depend in a fairly general fashion on time, we prove that such waves exist as soon as their speed is sufficiently large
in a sense. When this speed is too small, transition waves do not exist anymore, this result holds without assuming periodicity in space. 
These necessary and sufficient conditions are proved to be optimal when the coefficients are periodic both in space and time. 
Our method is quite robust and extends to general non-periodic space-time heterogeneous coefficients, showing that transition waves solutions of the nonlinear equation exist as soon as 
one can construct appropriate solutions of a given linearized equation. 

\vspace{0.2in}\noindent \textbf{Key words}: Fisher-KPP equation, reaction-diffusion, travelling waves, generalized transition waves, generalized principal eigenvalues. \\
\vspace{0.1in}\noindent \textbf{2010 Mathematical Subject Classification}: 35B40, 35K57, 35B51, 35K10, 35P05.

\end{abstract}

\setstcolor{red}

The research leading to these results has received funding from the European
Research Council under the European Union's Seventh Framework Programme
(FP/2007-2013) / ERC Grant Agreement n.321186 - ReaDi -Reaction-Diffusion
Equations, Propagation and Modelling. Luca Rossi was partially supported by 
GNAMPA-INdAM. 


\section{Introduction}


We are concerned with transition waves solutions of the space-time heterogeneous
reaction-diffusion equation
\begin{equation}\label{eq:princip}
\partial_t u-\Tr(A(x,t)D^2u)+q(x,t)\. Du=f(x,t,u),\quad x\in\R^N,\
t\in\R.
\end{equation}
Here $D$ and $D^2$ denote respectively the gradient and the Hessian with
respect to the space variables.
We assume that the terms in the equation are periodic in $x$, with the same
period.
The matrix field $A$ is uniformly elliptic and the nonlinearity
$f(x,t,\.)$ vanishes at $0$ and $1$. The steady states $0$ and $1$ are
respectively unstable and stable.
\\

When the coefficients do not depend on $(x,t)$, equation (\ref{eq:princip})
becomes a classical homogeneous monostable reaction-diffusion equation. 
The pioneering works on such equations are due to Kolmogorov, Petrovski and
Piskunov \cite{KPP} and Fisher \cite{Fisher} in the 30's, when $f(u) = u(1-u)$. 
They investigated the existence of travelling wave solutions,
that is, solutions of the form $u(x,t) = \phi (x\cdot e -ct)$, with $\phi
(-\infty) = 1$, $\phi (+\infty) = 0$, $\phi>0$. The quantity  
$c\in\R$ is the speed of the wave and $e\in\mathbb{S}^{N-1}$ is its direction.
Kolmogorov, Petrovski and Piskunov \cite{KPP} proved that when $A=I_N$, $q\equiv
0$ and $f=u(1-u)$,
there exists $c^*>0$ such that \eq{princip} admits travelling waves of
speed $c$ if and only if $c\geq c^*$. This property was extended to more
general 
monostable nonlinearities by Aronson and Weinberger \cite{AronsonWeinberger}.
The properties (uniqueness, stability, attractivity, decay at infinity) of these
waves have been extensively studied 
since then. 
\\



An increasing attention has been paid to heterogeneous reaction-diffusion since
the 2000's. In particular, the existence of appropriate generalizations of  
travelling waves solutions has been proved for various classes of
heterogeneities such as shear \cite{BerestyckiNirenberg}, time periodic
\cite{Alikakos},  
space periodic  \cite{BerestyckiHamel, BHR2, Xinperiodic}, space-time periodic
\cite{NolenRuddXin, Nadinptf}, time almost periodic \cite{Shenap2} and time
uniquely ergodic \cite{Shenue} ones,  
under several types of hypotheses on the nonlinearity. 
Now, the topical question is to understand whether reaction-diffusion equations
with general heterogeneous coefficients admit wave-like solutions or not.  
A generalization of the notion of travelling waves has been given by Berestycki
and Hamel \cite{BerestyckiHamelgtw, BerestyckiHamelgtw2}.

\begin{definition}\label{def:gtw}\cite{BerestyckiHamelgtw, BerestyckiHamelgtw2}
A {\em generalized transition wave} (in the direction
$e\in S^{N-1}$) is a positive time-global solution $u$ of \eq{princip} 
such that there exists a function $c\in L^\infty(\R)$ satisfying
\Fi{limits}
\lim_{x\.e\to-\infty} u\Big(x+e\int_0^t c(s)ds,t\Big)=1,\qquad
\lim_{x\.e\to+\infty} u\Big(x+e\int_0^t c(s)ds,t\Big)=0,
\Ff
uniformly with respect to $t\in\R$
\footnote{\ For a given function $g=g(x,t)$, the condition
$\lim_{x\.e\to\pm\infty}g(x,t)=l$
uniformly with respect to $t\in\R$ means that
$$\lim_{r\to+\infty}\sup_{\pm x\.e>r,\ t\in\R}|g(x,t)-l|=0.$$}. 
The function $c$ is called the {\em speed} of the generalized transition wave
$u$, and $\phi(x,t):=u(x+e\int_0^t c(s)ds,t)$ is the associated {\em profile}.
\end{definition}

The profile of a generalized transition wave satisfies
$$\lim_{x\.e\to-\infty}\phi(x,t)=1,\quad
\lim_{x\.e\to+\infty}\phi(x,t)=0,\quad
\text{uniformly with respect to }t\in\R.$$
It is clear that any perturbation of $c$ obtained by adding a function with
bounded integral is still a speed of $u$, with a different
profile. 
Reciprocally, if $\tilde{c}$ is another speed associated with $u$, then it is 
easy to check that 
$t\mapsto\int_0^t(c-\tilde{c})$ is bounded.
Obviously, all the notions of waves used previously when the coefficients belong
to particular classes of heterogeneities can be viewed as transition waves. 

The existence of such waves have been proved for one-dimensional space
heterogeneous reaction-diffusion with ignition-type nonlinearities (that is,
$f(x,u) = 0$ if $u\in [0,\theta)\cup \{1\}$ and $f(x,u)>0$
if $u\in (\theta,1)$) in parallel ways by Nolen and Ryzhik \cite{NolenRyzhik}
and Mellet, Roquejoffre and Sire \cite{MelletRoquejoffreSire}, and their
stability was proved in \cite{MelletNolenRoquejoffreRyzhik}. 
For space heterogeneous monostable nonlinearities, when $f(x,u) >0$
when $u\in (0,1)$ and $f(x,0) = f(x,1) = 0$, transition waves might not exist
\cite{NRRZ} in general. 
Let mention that this justified the introduction of the alternative notion of
{\em critical travelling wave} by the first author in \cite{NadinCTW} for
one-dimensional equations. 
Some existence results have also been obtained by Zlatos for partially periodic multi-dimensional
equations of ignition-type \cite{Zlatosdisordered}. 

When the coefficients only depend on $t$ in a general way, the existence of
transition waves was first proved by Shen for bistable nonlinearities
\cite{Shenbistable} (that is, nonlinearities vanishing at $u=0$ and $u=1$ 
but negative near these two equilibria) and for monostable equations with time
uniquely ergodic coefficients \cite{Shenue}. The 
case of general time heterogeneous monostable equations was investigated by the
two authors in \cite{NR1}. 
As observed in \cite{NR1}, the notions of least and upper mean play a crucial
role in such frameworks.

\begin{definition}\label{def:mean}
The {\em least mean} (resp.~the {\em upper mean}) {\em over} $\R$
of a function $g\in L^\infty(\R)$ is given by
$$\lfloor g\rfloor:=\lim_{T\to+\infty}\,\inf_{t\in\R}
\frac{1}{T}\int_t^{t+T}g(s)ds,\qquad\Big(\text{resp. }
\lceil g\rceil:=\lim_{T\to+\infty}\,\sup_{t\in\R}
\frac{1}{T}\int_t^{t+T}g(s)ds\ \Big).$$
\end{definition}

As shown in Proposition 3.1 of \cite{NR1}, the definitions of $\lm{g}$, 
$\um{g}$ do not change if one replaces $\lim_{T\to+\infty}$ with $\sup_{T>0}$ 
and $\inf_{T>0}$ respectively in the above expressions; this shows that 
$\lm{g}$, $\um{g}$ are well defined for any $g\in L^\infty(\R)$.
Notice that $g$ admits a uniform mean $\langle g\rangle$, that is
$\langle g\rangle:=\lim_{T\to+\infty}\frac{1}{T}\int_t^{t+T}g(s)ds$ exists
uniformly with respect to $t\in\R$, if and only if 
$\lfloor g\rfloor=\lceil g\rceil =\langle g\rangle$. This is the case in
particular when the coefficients are uniquely ergodic. 

Note that if $c$ and $\tilde{c}$ are two speeds associated with the same wave 
$u$, then $c-\tilde{c}$ has a bounded integral and thus 
$\lfloor c\rfloor = \lfloor \tilde{c} \rfloor$.

The two authors proved in \cite{NR1} that when $A\equiv I_N$, $q\equiv 0$ and
$f$ only depends on $(t,u)$ and is concave and positive with respect to $u\in
(0,1)$, there exists a speed $c_*>0$
such that for all $\gamma>c_*$ and $|e|=1$, equation (\ref{eq:princip}) admits a
generalized transition wave with speed $c=c(t)$ in the direction $e$ such that
$\lfloor c\rfloor = \gamma$, while no such waves exist 
when $\gamma<c_*$. 

\bigskip

When the coefficients not only depend on $t$ in a general way but also on $x$
periodically, some of the above results have been extended. 
Assuming in addition that the coefficients are 
uniquely ergodic and recurrent with respect to $t$ and that $A\equiv I_N$, Shen
\cite{Shenueper}
proved the existence of a speed $c_*$ such that for all $\gamma>c_*$, there
exists a 
generalized transition waves for monostable equations with speed $c$  

The case of space periodic and time general monostable equations was first
studied by the second author and Ryzhik \cite{RossiRyzhik}, 
under the additional assumption that the dependences in $t$ and $x$ are
separated, in the sense that $A$ and $q$ only depend on $x$, periodically,
while $f$ only depends on $(t,u)$. 
They proved both the existence of generalized transition waves of speed $c$ such
that $\lfloor c\rfloor >c_*$ and the non-existence of such waves with $\lfloor
c\rfloor <c_*$. 
Moreover, they provided a more general non-existence result, without assuming
that the dependence on $x$ of $A$ and $q$ is periodic. 

The aim of the present paper is to consider the general case
of coefficients depending on both $x$ and $t$. As in \cite{RossiRyzhik}, we assume
the periodicity in $x$ only for the existence result.


\section{Hypotheses and results}

\subsection{Statement of the main results}

Throughout the paper, the terms in \eq{princip} will always be assumed to
satisfy the following (classical) regularity hypotheses: 
\Fi{hyp-A}
\begin{cases}
A\text{ is symmetric, uniformly continuous},\\
\ex0<\ul\alpha\leq\ol\alpha,\quad\forall (x,t)\in\R^{N+1},
\quad \ul\alpha I\leq A(x,t)\leq\ol\alpha I,
\end{cases}
\Ff
\Fi{hyp-q}
q\,\text{ is bounded and uniformly continuous on }\R^{N+1},
\Ff
\Fi{hyp-f}
\begin{cases} 
f \,\text{ is a Caratheodory function on }\R^{N+1}\times[0,1],\\
\exists\delta>0,\ f(x,t,\.)\in W^{1,\infty}([0,1])\cap C^1([0,\delta)),
\text{ uniformly in }(x,t)\in\R^{N+1}.\\
\end{cases}
\Ff
The assumption that $q$ is uniformly continuous is a technical hypothesis used 
in the proofs to pass to the limit in sequences of translations of the 
equation. It could be replaced by $\dv q=0$.
We further assume that $f$ is of monostable type, 
$0$ being the unstable equilibrium and $1$ being the stable one. Namely,
\Fi{f(0)}
\forall (x,t)\in\R^{N+1},\quad
f(x,t,0)=0,
\Ff 
\Fi{f(1)}
\forall (x,t)\in\R^{N+1},\quad f(x,t,1)=0,
\Ff
\Fi{f>0} \forall u\in (0,1),\quad \inf_{(x,t)\in \R^{N+1}}f(x,t,u)>0.\Ff
In order to derive the existence result, we need some additional hypotheses. The
first one is the standard KPP condition:
\Fi{hyp-KPP}
\forall (x,t)\in\R^{N+1},\ u\in[0,1],\quad f(x,t,u)\leq \mu (x,t) u,
\Ff
where, here and in the sequel, $\mu$ denotes the function defined as follows:
$$\mu(x,t):=\partial_u f(x,t,0).$$ 
Conditions \eq{f>0}, \eq{hyp-KPP} imply that $\inf\mu>0$.
The second condition is
\Fi{C1gamma}
\exists C>0,\ \delta,\nu\in(0,1],\quad \forall x\in\R^N,\ t\in\R,\ u\in
(0,\delta),\quad 
f(x,t,u)\geq \mu (x,t) u-Cu^{1+\nu}.
\Ff
Note that a sufficient condition for \eq{C1gamma} to hold is 
$f(x,t,\.)\in C^{1+\nu}([0,\delta])$, uniformly with respect to $x,t$.
%
%
The last condition for the existence result is 
\Fi{periodic}
\exists l=(l_1,\dots,l_N)\in\R_+^N,\quad
\forall t\in\R,\ u\in(0,1),\quad
A,\ q,\ f\ \text{ are $l$-periodic in }x,
\Ff
where a function $g$ is said to be $l$-periodic in $x$ if it satisfies
$$\forall j\in\{1,\dots,N\},\
\forall x\in\R^N,\quad g(x+l_je_j)=g(x),$$
$(e_1,\dots,e_N)$ being the canonical basis of $\R^N$.

When we say that a function is a (sub, super) solution of \eq{princip} we
always mean that it is between $0$ and $1$. 
We deal with strong solutions whose derivatives $\partial_t$, $D$, $D^2$ belong
to some $L^p(\R^{N+1})$, $p\in(1,\infty)$. Many of our statements and
equations,
such as \eq{princip}, are understood
to hold a.e., even if we omit to specify it, and $\inf$, $\sup$ are used in
place of $\essinf$, $\esssup$.
\\


The main results of this paper consist in a sufficient and a necessary
conditions for the
existence of generalized transition waves, expressed in terms of their speeds. 

\begin{theorem}\label{thm:ex}
Under the assumptions \eq{hyp-A}-\eq{periodic}, for all $e \in
S^{N-1}$, there exists $c_*\in\R$ such
that for every $\gamma>c_*$,
there is a generalized transition wave in the direction $e$ with a speed
$c$ such that $\lfloor
c\rfloor= \gamma$.
\end{theorem}

The minimal speed $c_*$ we construct is explicitly given by (\ref{eq:defc}),
(\ref{eq:defLambda}) and (\ref{eq:defc*}). 
A natural question is to determine whether our construction gives an optimal
speed or not, that is, do generalized transition waves 
with speed $c$ such that $\lfloor c\rfloor<c_*$ exist? 
One naturally starts with checking if our $c_*$ coincides with
the optimal speed known to exist in some particular cases, such as space-time
periodic or space independent.
In Section \ref{sec:method} we show that this is the case.
The answer in the general, non space-periodic, case is only partial. It is 
contained in the next
theorem, where, however, we can relax the monostability 
hypotheses \eq{f>0}-\eq{hyp-KPP} by
\Fi{mu>0}
\left\lfloor\inf_{x\in\R^N}\mu(x,\.)\right\rfloor>0,
\Ff 
and we can drop \eq{f(1)}, \eq{C1gamma} as well as \eq{periodic}.
We actually need an extra regularity assumption on $A$:
\Fi{hyp-Aextra}
\text{$A$ is uniformly H\"older-continuous in $x$, uniformly with 
respect to $t$}.
\Ff
This ensures the validity of some a 
priori Lipschitz estimates quoted from \cite{PP}, that will be needed in the 
sequel. It is not clear to us if such estimates hold without 
\eq{hyp-Aextra}.

\begin{theorem}\label{thm:nex}
Under the assumptions \eq{hyp-A}-\eq{f(0)}, \eq{mu>0}-\eq{hyp-Aextra}, for all 
$e
\in S^{N-1}$, there exists $c^*\in\R$
such
that if $c$ is the speed of a generalized transition wave in the direction $e$
then $\lm{c}\geq c^*$.
\end{theorem}
We point out that no spatial-periodicity condition is assumed in the previous 
statement.
In order to prove \thm{nex} we derive a characterization of the least mean
- Proposition \ref{pro:nex} below - that we believe being of independent
interest.
The definition of $c^*$ is given in Section \ref{sec:Nonex}.
Of course, $c^*\leq c_*$ if the hypotheses of both Theorems \ref{thm:ex} and 
\ref{thm:nex} are fulfilled. We do not know if, in general,
$c_*=c^*$, that is, if the speed $c_*$ is minimal, in the 
sense that 
there does not exist any wave with a speed having a smaller least mean. 
When the coefficients are periodic in space and time or only depend on time, we could identify the speed $c_*$ more explicitly (see Section \ref{sec:method} below). 
Indeed, we recover in these frameworks some characterizations of the speeds identified in earlier papers \cite{Nadinptf, NR1, RossiRyzhik}, which were proved 
to be minimal. In the general framework, we leave this question open.

Finally, we leave as an open problem
the case $\lfloor c\rfloor=c_*$, for which we believe that 
generalized transition waves still exist.


\subsection{Optimality of the monostability assumption}\label{sec:optimality}

The assumption \eq{f>0} implies that $0$ and $1$ are respectively unstable
and stable. Let us discuss the meaning and the optimality of this hypothesis,
which might seem strong. 
Actually, as we do not make any additional assumption on the coefficients, we
can consider much more general asymptotic states $p_-=p_-(x,t)<p_+=p_+ (x,t)$
in place of $0$ and $1$,
and try to construct generalized transition waves $v$ connecting $p_-$ to $p_+$.
Indeed, if $p_\pm$ are solutions to \eqref{eq:princip}, with $p_+-p_-$ bounded
and having positive infimum, then the change of variables
$$u(x,t) := \frac{v(x,t) - p_- (x,t)}{p_+ (x,t) - p_- (x,t)}$$ leads
to an equation of the same form, with reaction term
$$\t f(x,t,u):=
\frac{f(x,t,up_++(1-u)p_-)-uf(x,t,p_+)-(1-u)f(x,t,p_-)}{p_+-p_-}.$$
The new equation admits the steady states $0$ and
$1$. Moreover, assuming 
that $u\mapsto f(x,t,u)$ is strictly concave, then
$\t f$ satisfies conditions \eq{f>0}, \eq{hyp-KPP}, the
latter following from the inequality
$$\forall u\in(0,1),\quad u(p_+-p_-)\partial_u f(x,t,p_-)\geq
f(x,t,up_++(1-u)p_-)-f(x,t,p_-).$$
This shows that, somehow, 
the concavity hypothesis of the nonlinearity with respect to $u$ is stronger, up
to some change of variables, than the positivity hypothesis of the nonlinearity.
 
Let us illustrate the above procedure with an explicit example where
$p_-\equiv0$.
Consider the equation 
\begin{equation} \label{eq:v} \partial_t v = \Delta v + \mu (x,t) v -v^2, \quad
x\in\R^N, \quad t\in\R,\end{equation}
with $\mu$ periodic in $x$, bounded and such that $\inf\mu>0$. The later
condition implies that the solution $0$ is linearly unstable (actually, it can
be relaxed by \eqref{eq:mu>0}, see the discussion below).
Then one can check that there is a time-global solution $p=p(x,t)$
which is bounded, has a positive infimum and is periodic in $x$. 
Let $u:= v/p$. This function satisfies
$$\partial_t u = \Delta u + 2\frac{\nabla p}{p} \cdot \nabla u + p(x,t) u
(1-u),$$
which is an equation of the form (\ref{eq:princip})  
for which
\eqref{eq:hyp-KPP}-\eqref{eq:periodic} hold, at least if, for instance, $\mu$ 
is uniformly H\"older-continuous, since then $\nabla p$ is bounded by 
Schauder's parabolic estimates, and $\inf p >0$.

Following this example, one can wonder whether \eq{f>0} is an optimal condition
(up to some change of variables) for the existence of transition waves.
It is well-known that other classes of nonlinearities, such as bistable or
ignition ones, could still give rise to transition waves (see for instance
\cite{BerestyckiHamel}). 
Thus, this question only makes sense if one reduces to the class of 
nonlinearities which are monostable, in a sense.
Let us assume that $f$ satisfies \eq{f(0)}, \eq{f(1)} and that $0$ is linearly
unstable, in the weak sense that \eq{mu>0} holds.
Then, using the properties of the least mean derived
in \cite{NR1}, one can construct arbitrarily small subsolutions $\ul u=\ul 
u(t)$ and thus, as $1$ is a positive solution, 
there exists a minimal solution $p$ of (\ref{eq:princip}) in the class of
bounded solutions with positive infimum. One could then check that 
our proof still works and gives rise to generalized transition waves connecting
$0$ to $p$. Indeed, condition \eq{f>0} only ensures that $p\equiv 1$. 
As a conclusion, the positivity hypothesis \eq{f>0} is not optimal: one could
replace it by \eq{mu>0} but then the generalized transition waves we construct
connect $0$ to the minimal time-global 
solution, which might not be $1$. 

Since for the existence of positive solutions it is sufficient to require
\eqref{eq:mu>0} rather than $\inf\mu>0$, one may argue that, in order to
guarantee that $1$ is the minimal time-global solution with positive infimum,
hypothesis \eq{f>0} could be relaxed by
\Fi{0inst}
\forall u\in (0,1),\quad \lm{\min_{x\in \R^N}f(x,\.,u)}>0.
\Ff
This is not true, as shown by the following example. Let $p\in C^1(\R)$ be a
strictly decreasing function such that $p(\pm\infty)\in(0,1)$.
Let $f$ satisfy $f(t,p(t))=p'(t)$.
It is clear that $f$ can be extended in such a way that \eqref{eq:0inst} holds;
however $p$ is a time-global solution of $\partial_ t u=f(t,u)$ with positive
infimum which is smaller than $1$. 

Finally, if $0$ is linearly stable, in the sense that
\Fi{mu<0}
\um{\sup_{x\in\R^N}\mu(x,\.)}<0
\Ff 
holds, and \eq{hyp-KPP} is satisfied, then there do not exist generalized
transition waves at all, and, more generally, solutions to the Cauchy problem
with bounded initial data converge uniformly to $0$ as $t\to\infty$. Indeed, as
an easy application of the property of the least (and upper) mean \eq{lm},
one can construct a supersolution $\ol u=e^{\sigma(t)-\e t}$, for some
$\sigma\in W^{1,\infty}(\R)$ and $\e>0$. The convergence to $0$ of
bounded solutions then follows from the comparison principle.


\subsection{Description of the method and application to particular cases}
\label{sec:method}

The starting point of the construction of generalized transition waves consists
in finding an explicit expression for the speed.
This is not a trivial task in the case of mixed space-time dependence
considered in this paper. We achieve it by an heuristic argument, that we
now illustrate.

Suppose that $u$ is a generalized transition wave in a direction $e\in
S^{N-1}$. Its tail at large $x\.e$ is close from being a solution 
of the linearized equation around $0$:
\Fi{linearized}
\partial_t u-\Tr(A(x,t)D^2u)+q(x,t)\. Du=\mu(x,t)u.
\Ff
It is natural to expect the tail of $u$ to decays exponentially. Thus, since
the equation is spatially periodic, we look for (the tail of) $u$ under the form
\Fi{form}
u(x,t)=e^{-\lambda x\.e}\eta_\lambda(x,t),\quad\text{ with
$\eta_\lambda$ positive and $l$-periodic in $x$}.
\Ff
Rewriting this expression as
$$u(x,t)={\rm exp}\left(-\lambda\left(x\.e-\frac1\lambda\ln\eta_\lambda(x,t)
\right)\right),$$
shows that the speed of $u$, namely, a function $c$ for which \eq{limits}
holds, should satisfy 
$$\left|\int_0^t
c(s)ds-\frac1\lambda\ln\eta_\lambda(x,t)\right|\leq C,$$
for some $C$ independent of $(x,t)\in\R^N\times\R$. Clearly, this can hold true
only
if the ratio between maximum and minimum of $\eta_\lambda(\.,t)$ is bounded
uniformly on $t$. This property follows from a Harnack-type inequality,
Lemma \ref{lem:eta} below, which is the keystone of our proof and actually the
only step where the periodicity in $x$ really plays a role. 
It would be then natural to define 
$c(t):=\frac1\lambda\frac d{dt}\ln\|\eta_\lambda(\.,t)\|_{L^\infty(\R)}$. The
problem is that we do not know if this function is bounded, since it is not clear whether $\partial_t\eta_\lambda\in L^\infty(\R^{N+1})$ or not. 
We overcome this difficulty by showing that there exists a Lipschitz
continuous function $S_\lambda$ such that 
\Fi{Slambda}
\exists\beta>0,\quad \forall t\in\R,\quad
\left|S_\lambda(t)-\frac1\lambda\ln\|\eta_\lambda(\.,t)\|_{
L^\infty(\R^N)}\right|\leq\beta.
\Ff
We deduce that the function $c$ defined (almost everywhere) by $c:=S_\lambda'$
is bounded and it is an admissible speed for the wave $u$. 
%
%
%
The method described above provides,  for any given
$\lambda>0$, a wave with speed $c=c_\lambda$ for the {\em linearized} equation
which decays with exponential rate $\lambda$.
It is known -for instance in
the case of constant coefficients- that only decaying rates which are
``not too fast'' are admissible for waves of the nonlinear
reaction-diffusion equation. In Section \ref{sec:subsol}, we identify a 
threshold rate $\lambda_*$. In the following section we construct
generalized transition waves for
any $\lambda<\lambda_*$, recovering with the least
mean of their speeds the whole interval $(\lm{c_{\lambda_*}},+\infty)$.
We do not know if the critical speed $c_*:=\lm{c_{\lambda_*}}$ is optimal, nor
if an optimal speed does exist. However, we show below that this is the case if
one applies
the above procedure to some particular classes of heterogeneities already investigated in the
literature.
\\

In the case where the coefficients are periodic in time too, the
class of admissible speeds has been characterized by the first author in
\cite{Nadinptf} (see also \cite{BHNa}).
Following the method described above, we see that an entire solution
of \eq{linearized} in the form \eq{form} is given by
$\eta_\lambda(x,t) = e^{k(\lambda)t} \varphi_\lambda (x,t)$, where
$\big(k(\lambda),
\varphi_\lambda\big)$ are the principal eigenelements of the 
problem\footnote{\
The properties of these eigenelements, which are unique (up to a multiplicative
constant in the case of $\vp_\lambda$) are described in \cite{Nadinptf} for
instance.}
\begin{equation} \label{eq:defklambda}\left\{ \begin{array}{l}
\partial_t\varphi_\lambda-\Tr(AD^2 \varphi_\lambda)
+(q+2\lambda Ae)D\varphi_\lambda-(\mu
+\lambda^2eAe+\lambda q\.e)\varphi_\lambda + k(\lambda)\varphi_\lambda=0 \hbox{
in } \R^N \times \R,\\
\varphi_\lambda>0,\\
\varphi_\lambda \hbox{ is periodic in } t \hbox{ and } x.
                                              \end{array}\right. 
\end{equation}
Actually, the uniqueness up to a multiplicative constant of solutions of 
\eq{linearized} in the form \eq{form}, provided by Lemma \ref{lem:eta}
(proved without assuming the time-periodicity), implies that $\eta_\lambda$ has 
necessarily this form.
Thus, $S_\lambda(t):=\frac{k(\lambda)}\lambda t$ satisfies
\eq{Slambda},
whence the speed of the wave for the linearized equation with decaying rate
$\lambda$ is $c_\lambda\equiv k(\lambda)/\lambda$. Since the $c_\lambda$
are constant (and then they have uniform mean), it turns out that the threshold 
$\lambda_*$ we obtain for the decaying rates coincides with the minimum 
point of
$\lambda\mapsto c_\lambda$ (see Remark \ref{rem:per} below). We eventually  derive the existence
of a generalized transition wave for any speed larger than
$c_*:=\min_{\lambda>0}k(\lambda)/\lambda$, which is exactly the sharp critical
speed for pulsating travelling fronts obtained in \cite{Nadinptf}. 
To sum up, our construction of the minimal speed $c_*$ is optimal in the
space-time periodic framework.
%
%
On the other hand, the definition of $c^*$  in Section
\ref{sec:Nonex} below
could also easily be identified to this speed and thus \thm{nex} implies
that there do not exist generalized transition waves with a speed $c$ such
that
$\lm{c}<\min_{\lambda
>0}k(\lambda)/\lambda$. We therefore recover also the non-existence
result for pulsating travelling fronts. Only the existence of fronts with
critical speed is not recovered.

In the case investigated by the authors in
\cite{NR1}, namely, when $A\equiv I_N$, $q\equiv0$ and $f$ does not depend
on $x$, one can easily check
that $\eta_\lambda(t) = e^{\int_0^t\mu
(s)ds+\lambda^2 t}$. As a function $S_\lambda$ we can simply take
$\frac1\lambda\ln\|\eta_\lambda(\.,t)\|_{
L^\infty(\R^N)}=\frac1\lambda\int_0^t\mu
(s)ds+\lambda t$, which is Lipschitz continuous. Whence $c_\lambda (t)
:= \lambda +
\frac{\mu (t)}{\lambda}$ is a speed of a wave with decaying rate $\lambda$.
In this case the critical decaying rate $\lambda_*$ is equal to
$\sqrt{\lm{\mu}}$ (see again Remark \ref{rem:per}) and
thus we have 
$c_* = 2\sqrt{\lm{\mu}}$. This is the same speed $c_*$ as in \cite{NR1}, which
was proved to be minimal. 

Under the assumptions made by Ryzhik and the second author in 
\cite{RossiRyzhik},
that is, $A$ and $q$ only depend on $x$ (periodically) and $f$ only depend
on $(t,u)$, 
the speeds $c_*$ derived in the present paper and in \cite{RossiRyzhik}
coincide, and thus it is minimal, in
the sense that there do not exist any generalized transition wave with a lower
speed. 

When $A\equiv I_N$ and $q$, $f$ are periodic in $x$ and uniquely ergodic in $t$,
then one can
prove that the same holds true for the function $\partial_t
\eta_\lambda/\eta_\lambda$ by 
uniqueness, and thus $\alpha \lm{c_\alpha}$
could be identified with the Lyapounov exponent
$\lambda (\alpha,\xi)$ used by Shen in \cite{Shenueper}, where $\xi$ is the
direction of 
propagation. We thus recover the same speed $c_*$ as in \cite{Shenueper} in
this
framework, which was not proved to be minimal since Shen did not 
investigate the nonexistence of transition waves with lower speed in
\cite{Shenueper}. 
Note that this identification is not completely obvious. However, as the
formalism of the present paper and \cite{Shenueper} are very different, we leave these computations to the reader. 
\\

Lastly, let consider the following example, where one could indeed construct
directly the generalized transition waves:
\begin{equation} \label{eq:ex}\partial_t u -\partial_{xx} u - q(t) \partial_x u =\mu_0 u (1-u),\end{equation}
with $q$ bounded and uniformly continuous and $\mu_0>0$. This equation 
satisfies assumptions \eq{hyp-A}-\eq{mu>0}. The change of variables 
$v(x,t):= u(x-\int_0^t q,t)$ leads to the classical homogeneous Fisher-KPP equation $\partial_t v -\partial_{xx} v = \mu_0 v(1-v)$. 
This equation admits travelling wave solutions of the form $v(x,t) =\phi_c
(x-ct)$, with $\phi_c (-\infty)=1$ and $\phi_c (+\infty)=0$, for all $c\geq
2\sqrt{\mu_0}$.  
Hence, equation (\ref{eq:ex}) admits generalized transition waves $u(x,t) =
\phi_c (x -ct +\int_0^t q,t)$ of speed $c-q(t)$ if and only if $c\geq
2\sqrt{\mu_0}$. That is, the set of least mean of admissible speeds is
$[2\sqrt{\mu_0}-\um{q},+\infty)$.
Computing $c_*$ in this case, one easily gets 
$$\eta_\lambda =\eta_\lambda (t) = e^{\lambda^2 t -\lambda \int_0^t q(s)ds+\mu_0 t}, \quad 
c_\lambda (t) = \lambda- q(t) +\mu_0/\lambda \quad \hbox{ and } \quad c_* =
2\sqrt{\mu_0}-\um{q}.$$ 
One could check that $c^*$ coincides with this value too, meaning that
Theorems \ref{thm:ex} and \ref{thm:nex} fully characterize the possible least means for
admissible speeds, except for the critical one.



\section{Existence result}                
\label{sec:ex}

Throughout this section, we fix $e\in S^{N-1}$ and we assume that
conditions \eq{hyp-A}-\eq{periodic} hold. 
Actually, condition \eq{f>0} could be weakened by \eq{mu>0}, except for the 
arguments in the very last part of the proof in 
Section \ref{sec:ccl}. As already mentioned in Section \ref{sec:optimality},
these arguments could be easily adapted to the case where \eq{f>0} is replaced 
by \eq{mu>0}, leading to transition waves connecting $0$ 
to the minimal solution with positive infimum. 



\subsection{Solving the linearized equation}

We focus on solutions with prescribed spatial exponential decay.

\begin{lemma}\label{lem:eta}
For all $\lambda>0$, the equation \eq{linearized}
admits a time-global solution of the form \eq{form}. Moreover, $\eta_\lambda$
is unique up to a multiplicative constant and satisfies, for all $t\in\R$,
$T\geq0$, 
\Fi{maxeta}
\max_{x\in\R^N}\eta_\lambda(x,t+T)\leq\max_{x\in\R^N}\eta_\lambda(x,t)\,
\exp\left((\ol\alpha\lambda+\sup_{\R^{N+1}}|q|)\lambda T+
\int_t^{t+T}\max_{x\in\R^N}\mu(x,s)ds\right),
\Ff
\Fi{mineta}
\min_{x\in\R^N}\eta_\lambda(x,t+T)\geq C\max_{x\in\R^N}\eta_\lambda(x,t)\,
\exp\left((\ul\alpha\lambda-\sup_{\R^{N+1}}|q|)\lambda T+
\int_t^{t+T}\min_{x\in\R^N}\mu(x,s)ds\right),
\Ff
with $C>0$ only depending on a constant bounding $|\lambda|$, $|l|$,
$\ul\alpha^{-1}$, $\ol\alpha$, $N$ and the $L^\infty$ norms of $\mu$ and $q$.
\end{lemma}

The function $(x,t) \mapsto e^{-\lambda x\.e}\eta_\lambda(x,t)$ is a solution of
the linearization of (\ref{eq:princip}) near the unstable equilibrium.  
We will show in the next section that it is somehow a transition wave solution
of the linearized equation, in the sense that it moves in the direction $e$ with
a certain speed. Due to hypothesis \eq{hyp-KPP}, we could use it as 
a supersolution of the nonlinear equation. Then, in Section \ref{sec:subsol}, in
order to construct an appropriate subsolution, we will need to restrict to
exponents $\lambda$ less than some threshold $\lambda_*$. We will eventually
derive the existence of transition waves in Section \ref{sec:ccl}.

As mentioned in Section \ref{sec:method}, Lemma \ref{lem:eta} is the only
point where the spatial periodicity hypothesis \eq{periodic} is used. If
the coefficients were depending in a general way on both
$x$ and $t$ and if one was able to construct a solution $\eta_\lambda$ of 
equation \eq{eta} for which there exists $C>0$ such that for all $T>0$,
$(x,t)\in\R^{N+1}$, one has:
\begin{equation} \label{eq:Harnackgen}\frac{1}{C} \|\eta_\lambda
(\cdot,t)\|_{L^\infty (\R^N)} e^{-CT} 
\leq \eta_\lambda (x,t+T)\leq C \|\eta_\lambda (\cdot,t)\|_{L^\infty (\R^N)}
e^{CT},\end{equation}
then the forthcoming other steps of the proof still apply and
it is possible to construct a generalized transition wave solution of the
nonlinear equation (\ref{eq:princip}). 
We describe this extension in Section \ref{sec:gen} below. 
It would be very useful to determine optimal conditions on the coefficients
enabling the derivation of a global Harnack-type inequality
(\ref{eq:Harnackgen})
for the linearized equation. We leave this question as an open problem. 
%

\begin{proof}[Proof of Lemma \ref{lem:eta}.]
The problem for $\eta_\lambda$ is
\Fi{eta}
\partial_t\eta_\lambda=\Tr(AD^2 \eta_\lambda)
-(q+2\lambda Ae)\.D\eta_\lambda+(\mu
+\lambda^2eAe+\lambda q\.e)\eta_\lambda,\quad x\in\R^N,\ t\in\R.
\Ff
We find a positive, $l$-periodic solution to \eq{eta} as the locally
uniform limit of (a subsequence of) solutions $\eta^n$ of the problem in
$\R^N\times(-n,+\infty)$,  
with initial datum $\eta^n(-n,\.)\equiv m_n$, where $m_n$ is a positive 
constant chosen in such a way that, say, $\sup_{x\in\R^N}\eta^n(0,x)=1$.

Let us show that any $l$-periodic solution $\eta_\lambda$ to \eq{eta} satisfies 
\eq{maxeta} and \eq{mineta}.
For given $t_0\in\R$, the function
$$\max_{x\in\R^N}\eta_\lambda(x,t_0)\,\exp\left((\ol\alpha\lambda^2+\sup_
{\R^{N+1}}|q|\lambda) (t-t_0)+
\int_{t_0}^t\max_{x\in\R^N}\mu(x,s)ds\right)$$
is a supersolution of \eq{eta} larger than $\eta_\lambda$ at time $t_0$. Since
$\eta_\lambda$ is bounded, we can apply the parabolic comparison principle and
derive \eq{maxeta}.
Let $\mc{C}$ denote the periodicity cell $\prod_{j=1}^N[0,l_j]$.
By parabolic Harnack's inequality (see, e.g., Corollary 7.42
in \cite{Lie}), we have that
\Fi{Harnack}
\forall t\in\R,\quad\max_{x\in\mc{C}}\eta_\lambda(x,t-1)\leq
\t C\min_{x\in\mc{C}}\eta_\lambda(x,t),
\Ff
for some $\t C>0$ depending on a constant bounding $|\lambda|$, $|l|$,
$\ul\alpha^{-1}$, $\ol\alpha$, $N$ and the $L^\infty$ norms of $\mu$ and $q$,
and
not on $t$.
On the other hand, the comparison principle yields, for $T\geq0$,
$$\min_{x\in\R^N}\eta_\lambda(x,t+T)\geq
\min_{x\in\R^N}\eta_\lambda(x,t)\,
\exp\left((\ul\alpha\lambda^2-\sup_{\R^{N+1}}|q|\lambda) T+
\int_t^{t+T}\min_{x\in\R^N}\mu(x,s)ds\right).$$
Combining this inequality with \eq{Harnack} we eventually derive
$$\min_{x\in\R^N}\eta_\lambda(x,t+T) \geq 
\t C^{-1}\max_{x\in\R^N}\eta_\lambda(x,t-1)\,
\exp\left((\ul\alpha\lambda^2-\sup_{\R^{N+1}}|q|\lambda) T+
\int_t^{t+T}\min_{x\in\R^N}\mu(x,s)ds\right),$$
from which \eq{mineta} follows by \eq{maxeta}. 

It remains to prove the uniqueness result.
Assume that \eq{linearized} admits two positive, $l$-periodic in
$x$ solutions $\eta^1$, $\eta^2$. As shown before, we know that they both
satisfy \eq{maxeta} and \eq{mineta}.
We first claim that there exists $K>1$ such
that
\Fi{trapped}
\forall t\in\R,\ x\in\R^N,\quad K^{-1}\eta^2(x,t)\leq\eta^1(x,t)\leq
K\eta^2(x,t).
\Ff
Let $h>0$ be such that $\eta^1\leq
h\eta^2$ at $t=0$. It follows that, for
$t\leq0$,
$\min_{x\in\R^N}\eta^1(x,t)\leq h\max_{x\in\R^N}\eta^2(x,t)$,
because otherwise the parabolic strong maximum principle would imply
$\eta^1>h\eta^2$ at $t=0$. Whence, applying \eq{mineta} with $T=0$
to both $\eta^1$ and $\eta^2$, we find a positive
constant $K$ such that
$$\forall t<0,\quad\max_{x\in\R^N}\eta^1(x,t)\leq
K\min_{x\in\R^N}\eta^2(x,t).$$
This proves the second inequality in \eq{trapped}, for $t<0$, whence for
all $t\in\R$ by the maximum principle.
The first inequality, with a possibly larger $K$, is obtained by exchanging the
roles of $\eta^1$ and $\eta^2$. Now, call
$$k:=\limsup_{t\to-\infty}\max_{x\in\R^N}
\frac{\eta^1(x,t)}{\eta^2(x,t)}.$$
We know from \eq{trapped} that $k\in[K^{-1},K]$. Consider a sequence
$\seq{t}$ such that 
$$\limn t_n=-\infty,\qquad\limn\max_{x\in\R^N}
\frac{\eta^1(x,t_n)}{\eta^2(x,t_n)}=k.$$
Define the sequences of functions $\seq{\eta^1}$, $\seq{\eta^2}$ as follows: 
$$\forall i\in\{1,2\},\ n\in\N,\quad\eta^i_n(x,t):=\frac{\eta^i(x,t+t_n)}
{\max_{y\in\R^N}\eta^1(y,t_n)}.$$
We deduce from \eq{maxeta} and \eq{mineta} that the $\seq{\eta^1}$
are uniformly bounded from above and uniformly bounded from below away from $0$
in, say, $\R^N\times[-2,2]$. The same is true for $\seq{\eta^2}$ by
\eq{trapped}. 
Thus, by parabolic estimates and periodicity
in $x$, the sequences $(\eta_n^i)_n$, $(\partial_t \eta_n^i)_n$, 
$(D\eta_n^i)_n$ and $(D^2\eta_n^i)_n$ converge, up to subsequences, in 
$L^p_{loc}(\R^{N+1})$.
Morrey's inequality yields that the sequences $(\eta_n^1)_n$ and $(\eta_n^2)_n$ 
converge locally uniformly to some functions $\t \eta^1$ and $\t \eta^2$ 
respectively. 

Call $A_n:=A (\cdot,\cdot+t_n)$, $q_n:=q (\cdot,\cdot +t_n)$,
$\mu_n:=\mu (\cdot,\cdot+t_n)$.
As $A$ and $q$ are uniformly continuous, $(A_n)_n$ and $(q_n)_n$ converge (up 
to subsequences) to
some functions $\t A$ and $\t q$ in $L^\infty_{loc} (\R^{N+1})$, whereas 
$(\mu_n)_n$ converges to some $\t\mu$ in the $L^\infty(\R^{N+1})$ 
weak-$\star$ topology.
Hence, passing to the weak $L^p_{loc}(\R^{N+1})$ limit in the equations 
satisfied by the $\seq{\eta^i}$, we get
$$
\partial_t\t\eta^i=\Tr(\t A D^2 \t\eta^i)
-(\t q +2\lambda \t A e)D\t \eta^i+(\t\mu
+\lambda^2e \t A e+\lambda \t q\.e)\t\eta^i,\quad x\in\R^N,\ t\in\R.
$$ 
Clearly, these equations hold almost everywhere because all the terms 
are measurable functions. That is, the $\t \eta^i$ are
strong solutions. Moreover,
$$\t\eta^1\leq k\t\eta^2,\qquad \max_{x\in\R^N}
\frac{\t\eta^1(x,0)}{\t\eta^2(x,0)}=k.$$
The strong maximum principle then yields $\t\eta^1\equiv k\t\eta^2$.
As a consequence, for any $\e>0$, we can find $n_\e\in\N$ such that, for $n\geq
n_\e$, $(k-\e)\eta^2_n<\eta^1_n<(k+\e)\eta^2_n$ at $t=0$. These inequalities
hold for all $t\geq0$, again by the maximum principle. Reverting to the
original functions we obtain
$(k-\e)\eta^2<\eta^1<(k+\e)\eta^2$ for $t\geq t_n$ and $n\geq n_\e$, from
which, letting $n\to\infty$ and $\e\to0^+$, we eventually infer that
$\eta^1\equiv k\eta^2$ for all $t\in\R$.
\end{proof}

In the particular case $T=0$, the inequality \eqref{eq:mineta} reads
\Fi{mineta0}
\min_{x\in\R^N}\eta_\lambda(x,t)\geq C\max_{x\in\R^N}\eta_\lambda(x,t).
\Ff
Notice that, in contrast with standard parabolic Harnack's inequality, the 
two sides are evaluated at the same time. This particular instance 
of \eqref{eq:mineta} will be sometimes used in the sequel.

Until the end of the proof of \thm{ex}, for $\lambda>0$, 
$\eta_\lambda$ stands for the (unique up to a multiplicative constant) function
given by Lemma \ref{lem:eta}.


\subsection{The speeds of the waves}

\begin{lemma}\label{lem:Slambda}
There is a uniformly Lipschitz-continuous function $S_\lambda:\R\to\R$ 
satisfying~\eq{Slambda}.
\end{lemma}

\begin{proof}
Properties \eq{maxeta}-\eq{mineta} yield the existence of a constant $\beta>0$
such that
$$\forall t\in\R,\ T\geq0,\quad
\big|\ln\|\eta_\lambda (\cdot,t+T)\|_{L^\infty (\R^N)}  -\ln\|\eta_\lambda
(\cdot,t)\|_{L^\infty (\R^N)}\big| \leq \beta(1+\lambda^2)T.$$
For all $n\in\mathbb{N}$, we define $S_\lambda$ on $[n,n+1]$ as the affine
function satisfying 
$$S_\lambda(n)=\frac{1}{\lambda}\ln\|\eta_\lambda 
(\cdot,n)\|_{L^\infty(\R^N)},\qquad
S_\lambda(n+1)=\frac{1}{\lambda}\ln\|\eta_\lambda 
(\cdot,n+1)\|_{L^\infty (\R^N)}.$$
Then for all $t\in (n,n+1)$, 
$$|S_\lambda '(t)| = \left|\frac{1}{\lambda}\ln\|\eta_\lambda
(\cdot,n+1)\|_{L^\infty (\R^N)}-\frac{1}{\lambda}\ln\|\eta_\lambda
(\cdot,n)\|_{L^\infty (\R^N)}\right| \leq \beta\frac{1+\lambda^2}\lambda.$$
Hence, $S_\lambda$ is uniformly Lipschitz-continuous over $\R$. 
Moreover, if $t\in [n,n+1]$, one has 
\[\begin{split}
\left|S_\lambda (t) - \frac{1}{\lambda}\ln\|\eta_\lambda (\cdot,t)\|_{L^\infty
(\R^N)}\right| &\leq 
\big|S_\lambda (t) -S_\lambda (n)\big|
+\frac{1}{\lambda}\big| \ln\|\eta_\lambda (\cdot,t)\|_{L^\infty
(\R^N)}-\ln\|\eta_\lambda (\cdot,n)\|_{L^\infty (\R^N)}\big|\\
&\leq 2\beta\frac{1+\lambda^2}\lambda.
\end{split}\] 
Hence, $t\mapsto S_\lambda (t) - \frac{1}{\lambda}\ln\|\eta_\lambda
(\cdot,t)\|_{L^\infty (\R^N)}$ is uniformly bounded over $\R$.  
\end{proof}

Owing to Lemma \ref{lem:Slambda}, the function $c_\lambda$ defined for (a.e.)
$t\in\R$ by
\begin{equation}\label{eq:defc} 
c_\lambda(t):= S_\lambda'(t),
\end{equation}
belongs to $L^\infty (\R)$. We will use it as a possible speed for a transition
wave to be constructed. 
\\

Let us investigate the properties of the least mean of the 
$(c_\lambda)_{\lambda>0}$. 
It follows from \eq{Slambda} that
\Fi{lmc}\lm{c_\lambda}=\frac1\lambda
\lim_{T\to+\infty}\inf_{t\in\R } \frac1T
\ln\frac{\|\eta_\lambda(\.,t+T)\|_{L^\infty(\R^N)}}{\|\eta_\lambda(\.,t)\|_{
L^\infty(\R^N)}}.
\Ff
Hence, by \eq{maxeta} and \eq{mineta}, we derive
\Fi{lmbounds}
\ul\alpha\lambda-\sup_{\R^{N+1}}|q|+\frac1\lambda\left\lfloor\min_{x\in\R^N}
\mu(x,
\.)\right\rfloor
\leq\left\lfloor{c_\lambda}\right\rfloor\leq\ol\alpha\lambda+\sup_{\R^{N+1}}|q|+
\frac1\lambda\left\lfloor\max_{x\in\R^N}\mu(x,\.)\right\rfloor.
\Ff
Analogous bounds hold for the upper mean:
\Fi{umbounds}
\ul\alpha\lambda-\sup_{\R^{N+1}}|q|+\frac1\lambda\um{\min_{x\in\R^N}\mu(x,\.)}
\leq\um{c_\lambda}\leq
\ol\alpha\lambda+\sup_{\R^{N+1}}|q|+\frac1\lambda\um{\max_{x\in\R^N}\mu(x,\.)}.
\Ff


We have
seen in Section \ref{sec:method} that, when the coefficients are periodic in
$t$, one can take $S_\lambda(t):=(k(\lambda)/\lambda) t$, whence 
$c_\lambda\equiv k(\lambda)/\lambda$. It follows that 
$\lambda c_\lambda=k(\lambda)$, and we know from the arguments in the 
proof of Proposition 5.7 part (iii) in \cite{BerestyckiHamel} that the function 
$k$ is convex.
In the general heterogeneous framework considered in the present paper, we  
use the same arguments as in \cite{BerestyckiHamel} to derive the Lipschitz
continuity of the function $\lambda\mapsto\lambda\lm{c_\lambda}$. If the 
functions $c_\lambda$ admit a uniform mean
then these arguments actually imply that $\lambda\mapsto\lambda\lm{c_\lambda}$ 
is convex, but we do not know if this is true in general.

\begin{lemma}\label{lem:c_l}
The functions $\lambda\mapsto\left\lfloor{c_\lambda}\right\rfloor$ and 
$\lambda\mapsto\left\lceil{c_\lambda}\right\rceil$ are
locally uniformly Lipschitz continuous on $(0,+\infty)$.
\end{lemma}

\begin{proof}
Fix $\Lambda>0$ and $-\Lambda\leq\lambda_0\leq\Lambda$. Let $\lambda_1$ be such
that $|\lambda_1-\lambda_0|=2\Lambda$. For $j=0,1$, the function $v_j(x,t):=
e^{-\lambda_j x\.e}\eta_{\lambda_j}(x,t)$ satisfies \eq{linearized}. Hence,
setting 
$v_j=e^{w_j}$, we find that
$$\partial_t w_j-\Tr\big(AD^2w_j\big)+q\. Dw_j=\mu
+\Tr\big(ADw_j\otimes Dw_j\big),\quad x\in\R^N,\ t\in\R.$$
For $\tau\in(0,1)$, the function $w:=(1-\tau)w_0+\tau w_1$ satisfies, for
$x\in\R^N$, $t\in\R$,
\[\begin{split}
\partial_t w-\Tr(AD^2w)+q\. Dw &=\mu
+\Tr\Big(A\big((1-\tau)Dw_0\otimes Dw_0+ \tau Dw_1\otimes Dw_1\big)\Big)\\
&\geq\mu+\Tr(ADw\otimes Dw).
\end{split}\]
As a consequence, $e^w$ is a supersolution of \eq{linearized} and then, since
$$e^{w(x,t)}=e^{-((1-\tau)\lambda_0+\tau\lambda_1)x\.e}
\eta_{\lambda_0}^{1-\tau}(x,t)
\eta_{\lambda_1}^\tau(x,t),$$
the function $\eta_{\lambda_0}^{1-\tau}\eta_{\lambda_1}^\tau$
is a supersolution of \eq{eta} with $\lambda=\lambda_\tau:=
(1-\tau)\lambda_0+\tau\lambda_1$. 
We can therefore apply the comparison principle between this function and 
$\eta_{\lambda_\tau}$ and derive, for $t\in\R$, $T>0$,
\[\begin{split}
\frac{\|\eta_{\lambda_\tau}(\.,t+T)\|_{L^\infty(\R^N)}}
{\|\eta_{\lambda_\tau}(\.,t)\|_
{L^\infty(\R^N)}} &\leq
\frac{\|\eta_{\lambda_0}^{1-\tau}\eta_{\lambda_1}^\tau(\.,t+T)\|_{L^\infty(\R^N)
}}{\min_{x\in\R^N}\eta_{\lambda_0}^{1-\tau}\eta_{\lambda_1}^\tau(x,t)}\\
&\leq\left(\frac{\|\eta_{\lambda_0}(\.,t+T)\|_{L^\infty(\R^N)
}}{\min_{x\in\R^N}\eta_{\lambda_0}(x,t)}\right)^{1-\tau}
\left(\frac{\|\eta_{\lambda_1}(\.,t+T)\|_{L^\infty(\R^N)
}}{\min_{x\in\R^N}\eta_{\lambda_1}(x,t)}\right)^\tau.
\end{split}\]
Whence, using the inequality \eq{mineta0} for $\eta_{\lambda_0}$ and
$\eta_{\lambda_1}$ (with the same $C$ depending on $\Lambda$), we obtain
\Fi{Q}
\frac{\|\eta_{\lambda_\tau}(\.,t+T)\|_{L^\infty(\R^N)}}
{\|\eta_{\lambda_\tau}(\.,t)\|_
{L^\infty(\R^N)}} \leq
C^{-1}\left(\frac{\|\eta_{\lambda_0}(\.,t+T)\|_{L^\infty(\R^N)}}
{\|\eta_{\lambda_0}(\.,t)\|_{L^\infty(\R^N)}}\right)^{1-\tau}
\left(\frac{\|\eta_{\lambda_1}(\.,t+T)\|_{L^\infty(\R^N)
}}{\|\eta_{\lambda_1}(\.,t)\|_{L^\infty(\R^N)}}\right)^\tau.
\Ff
Consider the function $\Gamma$ defined by $\Gamma(\lambda):= 
\lambda\lm{c_\lambda}$. It follows from \eq{lmc} and \eq{Q} that
$$\Gamma(\lambda_\tau)\leq\lim_{T\to+\infty}
\inf_ { t\in\R }
\frac1T\left((1-\tau)\ln\frac{\|\eta_{\lambda_0}(\.,t+T)\|_{L^\infty(\R^N)}}
{\|\eta_{\lambda_0}(\.,t)\|_{L^\infty(\R^N)}}+
\tau\ln\frac{\|\eta_{\lambda_1}(\.,t+T)\|_{L^\infty(\R^N)}}
{\|\eta_{\lambda_1}(\.,t)\|_{L^\infty(\R^N)}}\right).$$
If the $(c_\lambda)_{\lambda>0}$ admit uniform mean, the above inequality and 
\eq{lmc} imply that $\Gamma$ is convex. Otherwise, we can only infer that
$$\Gamma(\lambda_\tau)\leq(1-\tau)\Gamma(\lambda_0)+
\tau\lambda_1\lceil c_{\lambda_1}\rceil.
$$
We have therefore shown that
$$\forall\tau\in(0,1),\quad
\Gamma(\lambda_\tau)-\Gamma(\lambda_0)\leq\tau
(\lambda_1\lceil c_{\lambda_1}\rceil-\lambda_0\lfloor c_{\lambda_0}\rfloor).$$
Thus, by \eq{lmbounds} and \eq{umbounds} there exists a constant $K>0$,
depending on $A$, $q$, $\mu$, such that
$$\forall\tau\in(0,1),\quad\Gamma(\lambda_\tau)-\Gamma(\lambda_0)\leq
K(\Lambda^2+1)\tau.$$
This proves the Lipschitz continuity of $\Gamma$ on $[-\Lambda,\Lambda]$,
because $|\lambda_\tau-\lambda_0|=2\Lambda\tau$, concluding the proof of the 
lemma.
%
%

The same arguments lead to the local Lipschitz continuity of $\lambda\mapsto\left\lceil{c_\lambda}\right\rceil$.

\end{proof}


\subsection{Definition of the critical speed} \label{sec:subsol}

In order to define the critical speed $c_*$, we introduce the following set:
\begin{equation} \label{eq:defLambda} \Lambda:=\{\lambda>0\ :\ \exists \ol k>0,\
\forall0<k<\ol k,\ \lm{c_\lambda-c_{\lambda+k}}>0\}.\end{equation}

\begin{lemma}\label{lem:Lambda}
There exists $\lambda_*>0$ such that $\Lambda=(0,\lambda_*)$. Moreover, the
function $\lambda\mapsto\lm{c_\lambda}$ is decreasing on $\Lambda$.
\end{lemma}

\begin{proof}
Fix $\lambda_0,\lambda_1>0$. For $\tau\in(0,1)$, we set
$\lambda_\tau:=(1-\tau)\lambda_0+\tau\lambda_1$. Taking the $\ln$ in
\eq{Q}, and recalling that $c_\lambda=S_\lambda'$ with $S_\lambda$ satisfying
\eqref{eq:Slambda} yields
$$\int_t^{t+T}[(1-\tau)\lambda_0 c_{\lambda_0}+\tau\lambda_1
c_{\lambda_1}-\lambda_\tau c_{\lambda_\tau}]\geq \ln C-
4\lambda_\tau\beta.$$
Whence,
$$\lambda_\tau\int_t^{t+T}(c_{\lambda_0}-c_{\lambda_\tau})\geq
\tau\lambda_1\int_t^{t+T}(c_{\lambda_0}-c_{\lambda_1})+\ln C-
4\lambda_\tau\beta.$$
Dividing both sides by $T$, taking the $\inf$ on $t\in\R$ and then the $\lim$ as
$T\to+\infty$, we derive
\Fi{moveR}
\forall\tau\in(0,1),\quad
\lm{c_{\lambda_0}-c_{\lambda_\tau}}\geq\tau\frac{\lambda_1}{\lambda_\tau}
\lm{c_{\lambda_0}-c_{\lambda_1}}.
\Ff
If, instead, we divide by $-T$, we get
\Fi{moveL}
\forall\tau\in(0,1),\quad
\lm{c_{\lambda_\tau}-c_{\lambda_0}}\leq\tau\frac{\lambda_1}{\lambda_\tau}
\lm{c_{\lambda_1}-c_{\lambda_0}}.
\Ff
Analogous estimates hold of course for the upper mean.
The characterization of $\Lambda$ follows from these inequalities,
by suitable choices of $\lambda_0$,
$\lambda_1$ and $\tau$. We prove it in four steps.

Step 1. {\em $\Lambda\neq\emptyset$.}\\
The first inequality in \eq{lmbounds}, together with \eq{mu>0}, yield
$$\lim_{\lambda\to0^+}\lm{c_\lambda-c_1}\geq\lim_{\lambda\to0^+}\lm{c_\lambda}
-\um{c_1}=+\infty.$$
There exists then $0<\lambda<1$ such that $\lm{c_\lambda-c_1}>0$. Applying
\eq{moveR} with $\lambda_0=\lambda$, $\lambda_1=1$,
we eventually infer that $\lm{c_\lambda-c_{\lambda+k}}>0$, for all
$0<k<1-\lambda$, that is, $\lambda\in\Lambda$.

Step 2. {\em $\Lambda$ is bounded from above.}\\
By \eq{lmbounds} we obtain
$$\lim_{\lambda\to+\infty}\lm{c_1-c_{\lambda}}\leq
\lm{c_1}-\lim_{\lambda\to+\infty}\lm{c_\lambda }=-\infty.$$
There exists then $\lambda'>1$ such that, for
$\lambda>\lambda'$, $\lm{c_1-c_{\lambda}}<0$. Hence, for $k>0$,
applying \eq{moveL} with $\lambda_0=\lambda+k$, $\lambda_1=1$ and
$\tau=k/(k+\lambda-1)$, we derive
$$\lm{c_\lambda-c_{\lambda+k}}\leq\frac k{(k+\lambda-1)\lambda}
\lm{c_1-c_{\lambda+k}}<0.$$
Namely, $\lambda\notin\Lambda$ and thus $\Lambda$ is bounded from above by 
$\lambda'$.

Step 3. {\em If $\lambda\in\Lambda$ then $(0,\lambda]\subset\Lambda$.}\\
Let $0<\lambda'<\lambda$ and $k>0$. Using first \eq{moveR} and then \eq{moveL}
we get
$$\lm{c_{\lambda'}-c_{\lambda'+k}}\geq \left(\frac
k{k+\lambda-\lambda'}\right)\left(\frac{\lambda+k}{\lambda'+k}\right)
\lm{c_{\lambda'}-c_{\lambda+k}}\geq\left(\frac{\lambda+k}{\lambda'+k}\right)
\frac\lambda{
\lambda'}\lm{c_{\lambda}-c_{\lambda+k}}.$$
Thus, $\lambda\in\Lambda$ implies $\lambda'\in\Lambda$.

Step 4. {\em $\sup\Lambda\notin\Lambda$.}\\
Let $\lambda^*:=\sup\Lambda$ and $k>0$. For all $n\in\N$, there exists
$0<k_n<1/n$ such that $\lm{c_{\lambda^*+1/n}-c_{\lambda^*+1/n+k_n}}\leq0$.
For $n$ large enough, we have that $1/n+k_n<k$ and then, by \eq{moveR},
$$0\geq\lm{c_{\lambda^*+1/n}-c_{\lambda^*+1/n+k_n}}\geq
\left(\frac{k_n}{k-1/n}\right)\left(\frac{\lambda^*+k}{\lambda^*+1/n+k_n}\right)
\lm{c_{\lambda^*+1/n}-c_{\lambda^*+k}}.$$
Whence, 
$$\lm{c_{\lambda^*}-c_{\lambda^*+k}}\leq
\lm{c_{\lambda^*+1/n}-c_{\lambda^*+k}}+\um{c_{\lambda^*}-c_{\lambda^*+1/n}}
\leq\um{c_{\lambda^*}-c_{\lambda^*+1/n}}.$$
Using the the analogue of \eq{moveL} for the upper mean, we can control the 
latter term as follows:
$$\um{c_{\lambda^*}-c_{\lambda^*+1/n}}\leq
\frac{1/n}{\lambda_*+2/n}\um{c_{\lambda^*/2}-c_{\lambda^*+1/n}}
\leq\frac{1/n}{\lambda_*+2/n}\left(\um{c_{\lambda^*/2}}-\lm{c_{\lambda^*+1/n}}
\right),$$
which goes to $0$ as $n\to\infty$ (recall that
$\lambda\mapsto\lm{c_{\lambda}}$ is continuous by Lemma \ref{lem:c_l}).
We eventually infer that $\lm{c_{\lambda^*}-c_{\lambda^*+k}}\leq0$, that is,
$\lambda^*\notin\Lambda$.

It remains to show that $\lambda\mapsto\lm{c_\lambda}$ is decreasing on
$\Lambda$. Assume by contradiction that there are
$0<\lambda_1<\lambda_2<\lambda^*$ such that
$\lm{c_{\lambda_1}}\leq\lm{c_{\lambda_2}}$. 
The function $\lambda\mapsto\lm{c_\lambda}$, being continuous, 
attains its minimum on $[\lambda_1,\lambda_2]$ at some $\lambda$. Since
$\lm{c_{\lambda_1}}\leq\lm{c_{\lambda_2}}$, we can assume that
$\lambda\in[\lambda_1,\lambda_2)$. The definition of $\Lambda$ implies that
there exists $\lambda'\in(\lambda,\lambda_2)$ such that $\lm{c_\lambda
-c_{\lambda'}}>0$. Whence 
$$\lm{c_{\lambda'}}\leq\lm{c_\lambda}+\um{c_{\lambda'}-c_\lambda}=
\lm{c_\lambda}-\lm{c_\lambda-c_{\lambda'}}<
\lm{c_\lambda}.$$
This is impossible.
\end{proof}

%
%

We are now in position to define the critical speed $c_*$:
\begin{equation} \label{eq:defc*}
c_*:=\lm{c_{\lambda_*}},
\end{equation}
where $\lambda_*$ is given in Lemma \ref{lem:Lambda}.

\begin{remark}\label{rem:per}
When the terms in \eqref{eq:princip} are periodic in time, resuming from
Section \ref{sec:method}, we know that the speeds $(c_\lambda)_{\lambda>0}$ are
constant and satisfy $c_\lambda\equiv k(\lambda)/\lambda$, where $k(\lambda)$ is
the \pe\ of problem \eqref{eq:defklambda}. Hence,
$$\lm{c_\lambda -c_{\lambda+\kappa}}=\frac{k(\lambda)}{\lambda} -
\frac{k(\lambda+\kappa)}{\lambda+\kappa}.$$
As $\lambda\mapsto k(\lambda)$ is strictly convex (see \cite{Nadinptf}) and, 
by \eq{lmbounds}, 
$$\lim_{\lambda\to+\infty}\frac{k(\lambda)}\lambda=+\infty,\qquad\lim_{
\lambda\to0^+ } k(\lambda)=\lim_{\lambda\to0^+}\lambda
c_\lambda\geq\lm{\min_{x\in\R^N}\mu(x,\.)}>0,
$$
straightforward convexity arguments yield that 
$\lambda_* $ given by Lemma \ref{lem:Lambda} is the unique minimizer of
$\lambda\mapsto k(\lambda)/\lambda$. Therefore, $c_*= \min_{\lambda
>0}k(\lambda)/\lambda$, 
which is known to be the minimal speed for pulsating travelling
fronts (see \cite{Nadinptf}).
\end{remark}

\subsection{Construction of a subsolution and conclusion of the
proof}\label{sec:ccl}

In order to prove \thm{ex}, we introduce a family of functions
$(\vp_\lambda)_{\lambda>0}$ which play the role of the spatial periodic
principal eigenfunctions in the time-independent case. 
For $\lambda>0$, let $\eta_\lambda$ be the function given by Lemma
\ref{lem:eta}, normalized by $\|\eta_\lambda(\.,0)\|_{L^\infty(\R^N)}=1$. We
define
$$\vp_\lambda(x,t):=e^{-\lambda S_\lambda(t)}\eta_\lambda(x,t).$$
By \eq{Slambda} and \eq{mineta0}, there exist two positive constants
$C_\lambda,\beta$ such that
\Fi{vpbounds}
\forall x\in\R^N,\ t\in\R,\quad C_\lambda\leq\vp_\lambda(x,t) \leq
e^{\lambda\beta}.
\Ff

We will make use of the following key property of the least mean, provided
by Lemma~3.2 of \cite{NR1}:
\Fi{lm}
\forall g\in L^\infty(\R),\quad
\lfloor g\rfloor=\sup_{\sigma\in
W^{1,\infty}(\R)}\inf_{t\in\R}(\sigma'+g)(t).
\Ff

\begin{proof}[Proof of \thm{ex}]
Fix $\gamma>c_*$. Since the function 
$\lambda\mapsto\left\lfloor{c_\lambda}\right\rfloor$
is continuous by Lemma \ref{lem:c_l} and
tends to $+\infty$ as $\lambda\to0^+$ by \eq{lmbounds}, and
$\Lambda=(0,\lambda_*)$ by Lemma \ref{lem:Lambda}, there exists
$\lambda\in\Lambda$ such that $\left\lfloor{c_\lambda}\right\rfloor=\gamma$.
The function $w$ defined by
$$w(x,t):=\min\left(e^{-\lambda x\.e}\eta_\lambda(x,t)\,,
\,1\right)$$
is a generalized supersolution of \eq{princip}.

In order to construct a subsolution, consider the constant $\nu$ in 
\eq{C1gamma}. By the definition of
$\Lambda$, there exists $\lambda<\lambda'<(1+\nu)\lambda$ such that
$\lm{c_\lambda-c_{\lambda'}}>0$.
We then set
$\psi(x,t):=e^{\sigma(t)-\lambda'(x\.e-S_\lambda(t)
+S_{\lambda'}(t))}\eta_{\lambda'}(x,t)$ ,
where $\sigma\in W^{1,\infty}(\R)$ will be chosen later.
We have that
$$\partial_t \psi
-\Tr(A(x,t)D^2\psi)+q(x,t)\.D\psi-\mu(x,t)\psi=
\left[\sigma'(t)+\lambda'(c_\lambda(t)-c_{\lambda'}(t))\right]\psi.$$
Since $\lm{\lambda'(c_\lambda-c_{\lambda'})}=\lambda'
\lm{c_\lambda-c_{\lambda'}}>0$, 
by \eq{lm} we can choose $\sigma\in
W^{1,\infty}(\R)$ in such a way
that $K:=\inf_{\R}(\sigma'+\lambda'(c_\lambda-c_{\lambda'}))>0$. Hence,
$$\partial_t \psi
-\Tr(A(x,t)D^2\psi)+q(x,t)\.D\psi\geq(\mu(x,t)+K)\psi,\quad
x\in\R^N,\ t\in\R.$$
We define:
$$v(x,t):=e^{-\lambda
x\.e}\eta_\lambda(x,t)-m\psi(x,t),$$
where $m$ is a positive constant to be chosen. By computation,
$$e^{-\lambda x\.e}\eta_\lambda(x,t)-m\psi(x,t)=
e^{-\lambda(x\.e-S_\lambda (t))}\left(\vp_\lambda(x,t)-m\vp_{\lambda'
}(x,t)
e^{\sigma(t)-(\lambda'-\lambda)(x\.e-S_\lambda(t))}\right).$$
Since $\vp_\lambda$, $\vp_{\lambda'}$ satisfy \eq{vpbounds} and $\sigma \in 
L^\infty (\R)$, 
it follows that, choosing $m$ large enough, $v(x,t)\leq0$ 
if $x\.e-S_\lambda (t)\leq0$, and that $v$ is less than $\delta\in(0,1]$ from 
\eq{C1gamma} everywhere. If $v(x,t)>0$, and then
$x\.e-S_\lambda (t)>0$, we see that
\[\begin{split}
\partial_t v-\Tr(A(x,t)D^2v)+q(x,t)\.Dv-\mu(x,t)v &\leq-mK\psi\\
&\leq-mK\psi\frac{v^{1+\nu}}{e^{-(1+\nu)\lambda x\.e}\eta_\lambda
^{1+\nu}}\\
&=-mK v^{1+\nu}\frac{\vp_{\lambda'}}{\vp_{\lambda}^{1+\nu}}
e^{\sigma(t)-(\lambda'-(1+\nu)\lambda)(x\.e-S_\lambda(t))}\\
&\leq-mK
v^{1+\nu}C_{\lambda'}e^{-(1+\nu)\lambda\beta}
\inf_{s\in\R}e^{\sigma(s)},
\end{split}
\]
where, for the last inequality, we have used \eq{vpbounds} and the fact that
$\lambda'<(1+\nu)\lambda$. As a consequence,
by hypothesis \eq{C1gamma}, for $m$ sufficiently large,
$v$ is a subsolution of \eq{princip} in the set where it is positive.

Using again \eq{vpbounds}, one computes:
$$\begin{array}{rcl}
v(x+S_\lambda (t)e,t) &= & e^{-\lambda x\cdot e}\Big( \vp_{\lambda}(x+ S_{\lambda}(t)e,t) - m\vp_{\lambda'}(x+S_{\lambda}(t)e,t) e^{\sigma(t)-(\lambda'-\lambda) x\cdot e}\Big)\\
&\geq & e^{-\lambda x\cdot e}\Big( C_{\lambda}- 
m C_{\lambda'}e^{\lambda'\beta +\|\sigma\|_{\infty}-(\lambda'-\lambda) x\cdot 
e}\Big).\\
\end{array}$$
Hence, taking $R$ large enough, one has:
$$\inf_{\su{x\cdot e= R}{t\in \R}}v(x+S_\lambda (t)e,t) \geq 
e^{-\lambda R}\Big( C_{\lambda}- m C_{\lambda'}e^{\lambda'\beta 
+\|\sigma\|_{\infty}-(\lambda'-\lambda) R}\Big)=:\omega\in(0,1).$$
Consequently, the function $\ul v$ defined by
$$\underline{v} (x,t) := \left\{ \begin{array}{ccl}
v(x,t)   &\hbox{ if }& x\cdot e \geq S_{\lambda}(t)+R,\\
\max\big(\omega,v(x,t)\big)   &\hbox{ if }& x\cdot e<S_{\lambda}(t)+R,
\end{array} \right. $$
is continuous and, because of \eq{f>0}, it is a generalized subsolution of 
\eq{princip}.
Moreover, since $v\leq w$ and $w(x+S_\lambda (t)e,t)\geq 
e^{-\lambda R}C_{\lambda}>\omega$ if 
$x\cdot e<R$, one sees that $\underline{v}\leq w$.
A solution $\underline{v}\leq u \leq w$ can therefore be obtained as the limit 
of (a
subsequence of) the solutions $\seq u$ of the problems
$$\left\{\begin{array}{l}
\partial_t u_n -\Tr(A(x,t)D^2 u_n)+q(x,t)\.D u_n=
f(x,t,u_n), \quad
x\in\R^N,\ t>-n\\
 u_n(x,-n)=w(x,-n), \quad x\in\R^N.
\end{array}\right.$$
The strong maximum principle yields $u>0$.
One further sees that
$$\lim_{x\.e\to+\infty} u(x+e\int_0^t
c_\lambda(s)ds,t)\leq\lim_{x\.e\to+\infty} w(x+e\int_0^t
c_\lambda(s)ds,t)\leq\lim_{x\cdot e\to +\infty} 
e^{-\lambda x\cdot e}\vp_\lambda(x,t)=0,$$
uniformly with respect to $t\in\R$. 
It remains to prove that 
$$\lim_{x\.e\to-\infty} u(x+e\int_0^t c_\lambda(s)ds,t)=1$$
holds uniformly with respect to $t\in\R$. Set
$$\vartheta:=\lim_{r\to-\infty}\inf_{\su{x\.e\leq r}{t\in\R}}
u(x+e\int_0^t c_\lambda(s)ds,t).$$
Our aim is to show that $\vartheta=1$. We know that $\vartheta\geq\omega>0$, 
because $u(x,t)\geq\ul v(x,t)\geq\omega$ if $x\.e<S_{\lambda}(t)+R$.
Let $\seq x$ in $\R^N$ and $(t_n)_{n\in\N}$ in $\R$ be such that
$$\limn x_n\.e=-\infty,\qquad
\lim_{n\to\infty}u\big(x_n+e\int_0^{t_n}c_\lambda(s)ds,t_n\big)=\vartheta.$$
For $n\in\N$, let $k_n\in\prod_{j=1}^N l_j\Z$ be such that 
$y_n:=x_n+e\int_0^{t_n}c_\lambda(s)ds-k_n\in\prod_{j=1}^N[0,l_j)$ and define
$v_n(x,t):=
u(x+k_n,t+t_n)$.
The functions $\seq v$ are solutions of 
$$\partial_t v_n -\Tr(A(x, t+t_n)D^2v_n)+q(x, t+t_n)\.Dv_n= f(x,t+t_n,v_n), \quad
x\in\R^N,\ t\in\R.$$
By parabolic estimates, one can show using the same types of arguments as in 
the proof of Lemma \ref{lem:eta} that $\seq v$ converges (up to subsequences)
locally uniformly to some function $v$ satisfying
$$\partial_t v -\Tr(\tilde{A}(x,t)D^2v)+\tilde{q}(x,t)\.Dv= g(x,t)\geq 0, \quad
x\in\R^N,\ t\in\R,$$
where $\tilde{A}$ and $\tilde{q}$ are the strong limits in $L^\infty_{loc} 
(\R^{N+1})$ and $g$ is the weak limit in $L^p_{loc}(\R^N\times\R)$ of (a 
subsequence
of) $A(x,t+t_n)$, $q(x,t+t_n)$ and $f(x,t+t_n,v_n(x,t))$ respectively, the 
inequality $g\geq 0$ coming from hypothesis
\eq{f>0}. Furthermore, letting $y$ be the limit of (a
converging
subsequence of) $(y_n)_{n\in\N}$, we find
that $v(y,0)=\vartheta$ and
$$\forall x\in\R^N,\ t\in\R,\quad
v(x,t)=\lim_{n\to\infty}u(x+x_n+e\int_0^{t_n}c_\lambda(s)ds-y_n,t+t_n)
\geq\vartheta.$$
As a consequence, the strong maximum principle
yields
$v=\vartheta$ in $\R^N\times(-\infty,0]$. In particular, $g=0$
in $\R^N\times(-\infty,0)$.
Using the Lipschitz continuity of $f(x,t,\.)$,
we then derive for all $(x,t) \in\R^N\times (-\infty,0)$:
$$\forall T>0,\quad
0=\lim_{n\to +\infty} f(x,t+t_n,v_n(x,t))=
\lim_{n\to +\infty} f(x,t+t_n,\vartheta)\geq \inf_{(x,t) \in\R^{N+1}}
f(x,t,\vartheta).$$
This, by \eq{f>0}, implies that either $\vartheta=0$ or $\vartheta=1$, whence 
$\vartheta=1$ because $\vartheta \geq \omega>0$.
\end{proof}


\subsection{A criterion for the existence of generalized transition waves in
space-time general heterogeneous media}\label{sec:gen}

As already emphasized above, our proof holds in more general media, without
assuming that the coefficients satisfy \eq{periodic}, that is, without the space
periodicity assumption. We then need to assume that the linearized equation
admits a family of solutions satisfying some global Harnack inequality. 
We conclude the existence part of the paper by stating such a result. We omit
its proof since one
only needs to check that the previous arguments still work. 

\begin{theorem} In addition to \eq{hyp-A}-\eq{C1gamma}, 
assume that there exists $\overline{\lambda}>0$ such that for all $\lambda\in
(0,\overline{\lambda})$, there exists a Lipschitz-continuous time-global solution 
$\eta_{\lambda}$ of 
$$\partial_t\eta_{\lambda}=\Tr(AD^2 \eta_{\lambda})
-(q+2\lambda Ae)D\eta_{\lambda}+(\mu+\lambda^2eAe+\lambda
q\.e)\eta_{\lambda},\quad x\in\R^N,\ t\in\R$$
satisfying
$$\frac{1}{C} \|\eta_\lambda
(\cdot,t)\|_{L^\infty (\R^N)} e^{-CT} 
\leq \eta_\lambda (x,t+T)\leq C \|\eta_\lambda (\cdot,t)\|_{L^\infty (\R^N)}
e^{CT},$$
for some $C=C(\lambda)>0$ and for all $T>0$, 
$(x,t)\in\R^{N+1}$.

Then there exists $\lambda_* \in (0,\overline{\lambda})$ such that for all $\gamma>c_*:= \lm{S_{\lambda_*}'}$,
where $S_\lambda$ is a Lipschitz continuous function satisfying 
\eq{Slambda},
there exists a generalized transition wave with speed $c_\lambda=S_\lambda'$
such that $\lm{c_\lambda} = \gamma$. 
\end{theorem}


\section{Non-existence result}
\label{sec:Nonex}
%
%
%

Our aim is to find bounded subsolutions to the linearized problem
\begin{equation}\label{linear}
\partial_t u-Tr\big(A(x,t)D^2 u\big)+q(x,t)\cdot Du
=\mu(x,t)u,\quad x\in\R^N,\ t\in\R,
\end{equation} 
in order to get a lower bound for the speed of travelling wave solutions.
We recall that no spatial-periodic condition is now assumed.
Looking for solutions of \eqref{linear} in the form
$u(x,t)=e^{-\lambda(x\.e-ct)}\phi(x,t)$, with $\lambda$ and $c$ constant, leads
to the equation
\begin{equation}\label{Plambda}
(\mc{P}_\lambda+c\lambda)\phi=0,\quad
x\in\R^N,\ t\in\R,
\end{equation} 
where $\mc{P}_\lambda$ is the linear parabolic operator defined by
$$\begin{array}{rl} \mc{P}_\lambda w:=&\partial_tw
-Tr \big(A(x,t)D^2 w\big) + \big( q(x,t) - 2 \lambda A(x,t) e \big) \cdot Dw \\
&-\big( \lambda^2 e A(x,t) e +\lambda q(x,t) \cdot e +\mu(x,t) \big) w.\\
\end{array}$$
We consider the generalized \pe\ introduced in \cite{BN2}:
\Fi{pe}
\kappa(\lambda):=\inf\Big\{k\in\R\ : \ \exists\vp,\ 
\inf\vp>0,\ \sup\vp<\infty,\ \sup|D\vp|<\infty,\  
\mc{P}_\lambda\vp\leq k\vp,\ \text{in }\R^N\times\R\Big\},
\Ff
where the functions $\vp$ belong to $L^{N+1}_{loc}(\R^{N+1})$, together with
their derivatives $\partial_t$, $D$, $D^2$ (and then the differential
inequalities are understood to hold a.e.).
This is the minimal regularity required for the maximum principle to apply, 
see, e.g. \cite{Lie}.

Taking $\vp\equiv1$ in the above definition we get, for $\lambda\in\R$,
\Fi{phi=1}\kappa(\lambda)\leq-\ul\alpha\lambda^2+\sup_{\R^{N+1}}|q| |\lambda|-
\inf_{\R^{N+1} }\mu.
\Ff
We now derive a lower bound for $\kappa(\lambda)$.
Assume by way of contradiction that there exists a function $\vp$ as in the 
definition of $\kappa(\lambda)$, associated with some $k$ satisfying
$$k<-\ol\alpha\lambda^2-\sup_{\R^{N+1}}|q| 
|\lambda|-\sup_{\R^{N+1} }\mu.$$
For $\beta>0$, the function $\psi(x,t):=e^{-\beta t}$ satisfies
$$\frac{\mc{P}_\lambda\psi}{\psi}\geq-\beta 
-\ol\alpha\lambda^2-\sup_{\R^{N+1}}|q| 
|\lambda|-
\sup_{\R^{N+1} }\mu.$$
Hence, $\beta$ can be chosen small enough in such a way that the latter term 
is larger than $k$, that is, $\mc{P}_\lambda\psi\geq k\psi$. The function 
$\psi$ is larger than $\vp$ for $t$ less than some $t_0$, whence $\psi\geq\vp$ 
for all $t$ by the comparison principle. It follows that $\vp\to0$ as 
$t\to+\infty$, which is impossible since $\vp$ is bounded from below away from 
$0$. This shows that $\kappa(\lambda)>-\infty$.

We can now define $c^*$ by setting
\Fi{c*}
c^*:=-\max_{\lambda>0}\frac{\kappa(\lambda)}\lambda.
\Ff
This definition is well posed if $\kappa(0)<0$ because
$\kappa(\lambda)/\lambda\to-\infty$ as $\lambda\to +\infty$ by
\eqref{eq:phi=1}, and we know from \cite{BN2} that
$\lambda\mapsto\kappa(\lambda)$ is
Lipschitz-continuous \footnote{\ The
coefficients are assumed to be H\"older continuous in \cite{BN2}, but
one can check 
that it does not matter in the proof of the continuity.}.
Let us show that \eq{mu>0} implies that $\kappa(0)<0$ and then that $c^*$ is
well defined and finite.
Writing a positive function $\vp$ in the form $\vp(t):=e^{-\sigma(t)}$, we see
that
$$\mc{P}_0\vp=-(\sigma'(t)+\mu(x,t))\vp
\leq-\left(\sigma'(t)+\inf_ { x\in\R }
\mu(x,t)\right)\vp.$$
Thus, \eq{lm}
implies that, for given $\e>0$, there exists $\sigma\in W^{1,\infty}(\R)$ 
such that 
$$\mc{P}_0\vp\leq-\left(\lm{\inf_{x\in\R^N}\mu(x,\.)}-\e\right)\vp.$$
Therefore, if \eq{mu>0} holds, taking $\e<\lm{\min_{x\in\R^N}\mu(x,\.)}$ we
derive $\kappa(0)<0$.

The proof of \thm{nex} proceeds in two steps. In the following section we show
that the average between $0$ and $+\infty$ of the speed of a wave cannot be
smaller than $c^*$. More precisely, we derive
\begin{proposition}\label{pro:nex}
Assume that \eq{hyp-A}-\eq{f(0)} hold and that $\kappa(0)<0$. Then, for any 
nonnegative supersolution $u$ of \eq{princip} such that there is $c\in
L^\infty(\R)$ satisfying \eq{limits}, there holds
$$\liminf_{t\to+\infty}\frac1t\int_0^t
c(s)ds\geq c^*:=-\max_{\lambda>0}\frac{\kappa(\lambda)}\lambda.$$
\end{proposition}
In this statement, the notion of (sub, super) solution is understood as in the
definition of $\kappa(\lambda)$: $u,\partial_t
u,Du,D^2u\in L^{N+1}_{loc}(\R^{N+1})$.
Notice that the least mean of a function is in general smaller than the average
between $0$
and $+\infty$. In the last section, 
we establish a general property of the least mean that allows
us to deduce \thm{nex} by applying
Proposition \ref{pro:nex} to suitable time translations of the original
problem.


\subsection{Lower bound on the mean speed for positive times}

%

We start with constructing subsolutions with a slightly varying 
exponential behaviour as $x\.e\to\pm\infty$. These will be then used to build a 
generalized subsolution with an arbitrary modulation of the
exponential behaviour. The term ``generalized subsolution'' refers to  a
function that, in a neighbourhood of each point, is obtained as the supremum
of some family of subsolutions. Then, using the fact that the generalized 
subsolutions 
satisfy the maximum principle, we will be able to prove Proposition 
\ref{pro:nex}.

\begin{lemma}\label{lem:modulation}
Let $c,\lambda\in\R$ be such that $\kappa(\lambda)+c\lambda<0$. There exists
then $\e>0$ and $M>1$ such
that, for any $z\in\R$, \eqref{linear}
admits a subsolution $\ul v$ satisfying
$$\text{if }x\.e-ct\geq z,\qquad\frac1Me^{-(\lambda+\e)(x\.e-ct)}\leq\ul
v(x,t)\leq Me^{-(\lambda+\e)(x\.e-ct)},$$
$$\inf_{z-1<x\.e-ct<z}\ul v(x,t)>0,$$ 
$$\text{if }x\.e-ct\leq z-1,\qquad\frac1Me^{-(\lambda-\e)(x\.e-ct)}\leq\ul
v(x,t)\leq Me^{-(\lambda-\e)(x\.e-ct)}.$$
\end{lemma}

\begin{proof}
By the definition of $\kappa(\lambda)$, there is a bounded function $\vp$
with positive infimum satisfying 
$$\mc{P}_\lambda\vp\leq k\vp,\quad x\in\R^N,\ t>T,$$
for some $k<-c\lambda$.
It follows that $v(x,t):=e^{-\lambda(x\.e-ct)}\vp(x,t)$ is a subsolution of
\eqref{linear}. Fix $z\in\R$ and consider a smooth function $\zeta:\R\to\R$
satisfying 
$$\zeta=\lambda-\e\text{ in }(-\infty,z-1],\qquad
\zeta=\lambda+\e\text{ in
}[z,+\infty),\qquad 0\leq\zeta'\leq3\e,\qquad|\zeta''|\leq h\e,$$
where $\e>0$ has to be chosen and $h$ is a universal constant.
We define the function $\ul v$ by setting $\ul
v(x,t):=e^{-(x\.e-ct)\zeta(x\.e-ct)}\vp(x,t)$.
Calling $\rho:=x\.e-ct$, we find that
\[\begin{split}
[\partial_t\ul v-a_{ij}(x,t)\partial_{ij}\ul
v+q_i(x,t)\partial_i\ul v-\mu(x,t)\ul v]e^{\rho}\leq
&(\mc{P}_\zeta+c\zeta)\vp\\
&+C[(1+\rho+\rho|\zeta|+\rho^2|\zeta'|)|\zeta'|+\rho|\zeta''|],
\end{split}\]
where $\zeta,\zeta',\zeta''$ are evaluated at $\rho$ and $C$ is a constant
 depending on $N,c$ and the $L^\infty$ norms of $a_{ij},q,\mu,\vp,D\vp$.
The second term of the above right-hand side is bounded by $H(\e)$, for some
continuous function $H$ vanishing at $0$. The first term satisfies
$$(\mc{P}_\zeta+c\zeta)\vp\leq (\mc{P}_\lambda+c\lambda)\vp+
C((\zeta-\lambda)+|\zeta^2-\lambda^2|)\leq(k+c\lambda)\vp+
C(\e+2|\lambda|\e+\e^2).$$
We thus derive
$$\partial_t\ul v-a_{ij}(x,t)\partial_{ij}\ul
v+q_i(x,t)\partial_i\ul v
-\mu(x,t)\ul v\leq e^{\rho}[(k+c\lambda)\vp+
C\e(1+2|\lambda|+\e^2)+H(\e)].$$
Since $k<-c\lambda$ and $\inf\vp>0$, $\e$ can be chosen small enough in such a
way that $\ul v$ is a subsolution of \eqref{linear}.
\end{proof}

\begin{lemma}\label{lem:subsol}
Let $\ul\lambda,\ol\lambda,c\in\R$ satisfy $\ul\lambda<\ol\lambda$ and 
$$\max_{\lambda\in[\ul\lambda,\ol\lambda]}
\left(\kappa(\lambda)+c\lambda\right)<0.$$
Then there exists a generalized, bounded subsolution $\ul v$ of
\eqref{linear} satisfying
$$
\lim_{r\to-\infty}\sup_{x\.e-ct<r}\ul
v(x,t)e^{\ul\lambda(x\.e-ct)} =0,\qquad
\lim_{r\to+\infty}\sup_{x\.e-ct>r}\ul
v(x,t)e^{\ol\lambda(x\.e-ct)}=0.
$$
\Fi{sub>0}
\forall r_1<r_2,\quad \inf_{r_1<x\.e-ct<r_2}\ul v(x,t)>0.
\Ff
\end{lemma}

\begin{proof}
For $\lambda\in[\ul\lambda,\ol\lambda]$, let $\e_\lambda$, $M_\lambda$ be the
constants given by Lemma
\ref{lem:modulation} associated with $c$ and $\lambda$. Call $I_\lambda$
the
interval $(\lambda-\e_\lambda,\lambda+\e_\lambda)$.
The family $(I_\lambda)_{\ul\lambda\leq\lambda\leq\ol\lambda}$
is an open covering of $[\ul\lambda,\ol\lambda]$. Let
$(I_{\lambda_i})_{i=1,\dots,n}$ be a finite subcovering and set for short
$\e_i:=\e_{\lambda_i}$, $M_i:=M_{\lambda_i}$. 
Up to rearranging the indices and extracting another subcovering if need be, we
can assume that
$$\forall i=1,\dots,n-1,\quad
\lambda_{i+1}-\e_{i+1}<\lambda_i-\e_i<\lambda_{i+1}+\e_{i+1}<\lambda_i+\e_i.$$
%
%
%
Let $v_1$ be the subsolution of \eqref{linear} given by
Lemma \ref{lem:modulation} associated with $\lambda=\lambda_1$ and
$z=0$.
Set $z_1:=0$, $k_1:=1$ and
$$k_2:=\frac{e^{(\lambda_2+\e_2-(\lambda_1-\e_1))(z_1-1)}}{M_1M_2}.$$
Consider then
the subsolution $v_2$ associated with $\lambda=\lambda_2$ and $z$ equal to some
value $z_2<z_1-1$ to be chosen.
We have that
$$\text{if
}x\.e-ct=z_1-1,\qquad\frac{v_1(x,t)}{v_2(x,t)}\geq\frac{k_2}{k_1},$$
$$\text{if }x\.e-ct=z_2,\qquad\frac{k_1v_1(x,t)}{k_2v_2(x,t)}\leq
(M_1M_2)^2e^{(\lambda_2+\e_2-(\lambda_1-\e_1))(z_2-z_1+1)}.$$
Since $\lambda_2+\e_2>\lambda_1-\e_1$, we can choose  $z_2 $  in such a way
that
the latter term is less than $1$.
By a recursive argument, we find some constants $(z_i)_{i=1,\dots,n}$ satisfying
$z_n<z_{n-1}-1<\cdots<z_1-1=-1$, such that the family of subsolutions
$(v_i)_{i=1,\dots,n}$ given by Lemma \ref{lem:modulation} associated with the
$(\lambda_i)_{i=1,\dots,n}$ and $(z_i)_{i=1,\dots,n}$
satisfies, for some positive $(k_i)_{i=1,\dots,n}$,
$$\forall i=1,\dots,n-1,\quad k_{i+1}v_{i+1}\leq k_i v_i\ \text{ if }
x\.e-ct=z_i-1,\quad k_{i+1}v_{i+1}\geq k_i v_i\ \text{ if }
x\.e-ct=z_{i+1}.$$
The function $\ul v$, defined by
$$\ul v(x,t):=\begin{cases}
               v_1(x,t) & \text{if }x\.e-ct\geq z_1\\
               \max(k_i v_i(x,t),k_{i+1}v_{i+1}(x,t)) &
\text{if }z_{i+1}\leq x\.e-ct<z_i\\
k_nv_n(x,t) & \text{if }x\.e-ct<z_n,
              \end{cases}$$
is a generalized subsolution of \eqref{linear} satisfying the desired
properties.
%
\end{proof}

\begin{proof}[Proof of Proposition \ref{pro:nex}]
Let $u$, $c$ be as in the statement of the proposition, and call
$\phi(x,t):=u(x+e\int_0^t c(s)ds,t)$.
Since $\phi(x,t)\to1$ as $x\.e\to-\infty$, uniformly with respect to $t\in\R$,
one can find $\rho\in\R$ such that 
$$\inf_{\su{x\cdot e <\rho}{t\in\R}}\phi(x,t)>0.$$ 
We make now use of Lemma 3.1 in
\cite{RossiRyzhik}, which, under the above condition, establishes a lower bound
for the exponential decay of an entire supersolution $\phi$ of a linear
parabolic equation (notice that the differential
inequality for $\phi$ can be written in linear form with a bounded zero order
term: $f(x,t,\phi)=[f(x,t,\phi)/\phi]\phi$). The result of \cite{RossiRyzhik}
implies the existence of a positive constant $\lambda_0$ such that
$$\inf_{\su{x\.e>\rho-1}{t\in\R}}\phi(x,t)e^{\lambda_0 x\.e}>0.$$
By the definition of $c^*$, the hypotheses of Lemma \ref{lem:subsol} are
fulfilled with
$\ul\lambda=0$, $\ol\lambda=\lambda_0$ and $c=c^*-\e$, for any given $\e>0$.
This is also true if one penalizes the nonlinear term $f(x,t,u)$ by subtracting
$\delta u$, with $\delta$ small enough, since this just raises the \pe s
$\kappa(\lambda)$ by $\delta$. Therefore, Lemma \ref{lem:subsol} provides a
function $\ul v$ such that, for $h>0$ small enough, $h\ul v$ is a subsolution
of \eq{princip}.
We choose $h$ in such a way that, together with the above property, $h\ul v
(x,0)<u(x,0)$. This can be done, due to the lower bounds of $u(x,0)=\phi(x,0)$,
because $\ul v$ is bounded and decays faster than $e^{-\lambda_0 x\.e}$ as
$x\.e\to+\infty$. Applying the parabolic comparison principle we eventually
infer that $h\ul v<u$ for all $x\in\R^N$, $t\geq0$. It follows that $u$
satisfies \eqref{eq:sub>0} with $c=c^*-\e$ for $t>0$. We derive, in particular,
$$0<\inf_{t>0}u((c^*-\e)te,t)=\inf_{t>0}u\Big(\Big[(c^*-\e)t-\int_0^t
c(s)ds\Big]e+e\int_0^tc(s)ds,t\Big),$$
which, in virtue of the second condition in \eq{limits}, implies that 
$$\limsup_{t\to+\infty}\Big((c^*-\e)t-\int_0^tc(s)ds\Big)<+\infty.$$
This concludes the proof due to the arbitrariness of $\e$.
\end{proof}


\subsection{Property of the least mean and proof of \thm{nex}}

Roughly speaking, the least mean of a function is the infimum of its averages in
sufficiently large intervals. We show that, in some sense, this infimum
is achieved up to replacing the function with an element of its {\em
$\omega$-limit} set. 
The $\omega$-limit (in the $L^\infty$ weak-$\star$ topology) of a function $g$,
denoted by $\omega_g$, is the set of
functions obtained as $L^\infty$ weak-$\star$ limits of translations
of $g$. 
\begin{proposition} \label{reformulation}
Let $g\in L^\infty(\R)$ and let $\omega_g$ denote its $\omega$-limit set (in the
$L^\infty$ weak-$\star$ topology). Then
$$\lm{g}=\min_{\t g\in\omega_g}\left(
\lim_{t\to+\infty}\frac{1}{t}\int_0^t \t g(s)\, ds\right).$$
\end{proposition}

\begin{proof}
We can assume without loss of generality that $\lm{g}=0$.
Clearly, any $\t g\in\omega_g$ satisfies $\lm{\t g}\geq\lm{g}$, whence
$$\liminf_{t\to+\infty}\frac{1}{t}\int_0^t \t g(s)\, ds\geq\lm{\t g}\geq
\lm{g}=0.$$
Our aim is to find a function $\t g\in\omega_g$ satisfying
\Fi{g*}\limsup_{t\to+\infty}\frac{1}{t}\int_0^t \t g(s)\, ds\leq0.\Ff
We claim that, for any $n\in\N$, there exists $t_n\in\N$ such that
$$\forall j=1,\dots,n,\quad n\int_{t_n}^{t_n+j}g(s)\, ds\leq j.$$
Assume by way of contradiction that this property fails for some $n\in\N$.
By the definition of least mean, for $K\in\N$ large enough, there is
$\tau\in\R$ such that 
$$\frac1K\int_\tau^{\tau+Kn}g(s)\, ds<\frac{1}2.$$
On the other hand, there is $j\in\{1,\dots,n\}$ such that 
$n\int_{\tau}^{\tau+j}g(s)\, ds>j$. Then, there is $h\in\{1,\dots,n\}$ such 
that 
$n\int_{\tau+j}^{\tau+j+h}g(s)\, ds>h$, hence 
$n\int_{\tau}^{\tau+j+h}g(s)\, ds>j+h$. We repeat this argument until we
find $k\in\{1,\dots,n\}$ such that 
$n\int_{\tau}^{\tau+Kn+k}g(s)\, ds>Kn+k$. From this we deduce that
$$\int_\tau^{\tau+Kn}g(s)\, ds>K+\frac kn-\int_{\tau+Kn}^{\tau+Kn+k}g(s)\, ds
>K-n\|g\|_{L^\infty(\R)}.$$
A contradiction follows taking $K>2n\|g\|_{L^\infty(\R)}$. The claim is proved.
The $L^\infty$ weak-$\star$ limit $\t g$ as $n\to\infty$ of (a subsequence of) 
$g(\.+t_n)$ satisfies the desired property. Indeed,
$$\forall j\in\N,\quad\int_0^j \t g(s)\, ds=\limn
\int_{t_n}^{t_n+j} g(s)\, ds=0,$$
from which \eq{g*} follows since $\t g$ is bounded.
\end{proof}

\begin{proof}[Proof of \thm{nex}]
Let $u$ be a generalized transition wave with speed $c$.
Proposition \ref{reformulation} yields that there exists $\tilde{c} \in
\omega_{c}$ such that 
\Fi{omegac}
\lm{c} = \lim_{T\to +\infty} \frac{1}{T}\int_0^T \tilde{c}(s)ds.
\Ff
The definition of $\omega_c$ gives a sequence $(t_n)_{n\in\N}$ in $\R$ such
that 
$c(\cdot +t_n) \rightharpoonup \tilde{c}$ as $n\to +\infty$ for the $L^\infty$
weak-$\star$ topology. 
For $n\in\N$, consider the functions
$$A_n (x,t):=A\Big(x+e\int_0^{t_n} c(s)ds, t +t_n\Big),\qquad
q_n (x,t):=q\Big(x+e\int_0^{t_n} c(s)ds, t +t_n\Big),$$
$$\mu_n (x,t):=\mu\Big(x+e\int_0^{t_n} c(s)ds, t +t_n\Big),
\qquad u_n (x,t):=u\Big(x+e\int_0^{t_n} c(s)ds, t +t_n\Big).$$
For any $\e\in(0,1)$ there exists $m\in(0,1)$ such that
$$\forall (x,t)\in\R^{N+1},\ u\in[0,1],\quad f(x,t,u)\geq(\mu(x,t)-\e)u(m-u).$$ 
It follows that the $u_n$ satisfy
$$\partial_t
u_n-\Tr(A_n(x,t)D^2u_n)+q_n(x,t) D u_n\geq(\mu_n(x,t)-\e)u_n(m-u_n),
\quad x\in\R^N,\ t\in\R.$$
On the other hand, the $L^p$ parabolic interior estimates ensure that the
sequences 
$\seq{\partial_t u}$, $\seq{Du}$, $\seq{D^2u}$ are bounded
in $L^p(Q)$ for all $p\in (1,\infty)$ and $Q\Subset\R^{N+1}$. Hence, by the 
embedding theorem,
$\seq{u}$ converges (up to subsequences) locally uniformly in $\R^{N+1}$ to some
function $\t u$, and the derivatives $\partial_t$, $D$, $D^2$ of the $\seq{u}$
weakly converge to those of $\t u$ in $L^p_{loc}(\R^{N+1})$.
Therefore, letting $\t A$, $\t q$ be the locally uniform limits of
(subsequences of) $\seq{A}$, $\seq{q}$ and $\t\mu$ be the $L^\infty$
weak-$\star$ limit of (a subsequence of) $\seq{\mu}$, we infer
that  $$\partial_t
\t u-\Tr(\t A(x,t)D^2\t u)+\t q(x,t) D\t u\geq(\t\mu(x,t)-\e)\t u(m-\t u),
\quad x\in\R^N,\ t\in\R.$$
Hence, $\t u$ is a supersolution of an equation of the type
\eq{princip} whose terms satisfy \eq{hyp-A}-\eq{hyp-f} and \eq{f(0)}.
Moreover, it is easily derived from the definition of the speed
$c$ and the
$L^\infty$ weak-$\star$ convergence to $\t c$, that
$\t u$ satisfies \eq{limits} with $c$ replaced by $\t c$, uniformly with
respect to $t\in\R$. 
In order to apply Proposition \ref{pro:nex} to the function $\t u$, we need to
show that $\t\kappa(0)<0$, where $\lambda\mapsto\t\kappa(\lambda)$ is defined
like $\lambda\mapsto\kappa(\lambda)$, but with $\t A$, $\t q$, $\t\mu-\e$ in
place of $A$, $q$,
$\mu$ respectively.
Namely, the $\kappa(\lambda)$ are the \pe s in the sense of \eq{pe} for the
operators $\t{\mc{
P}}_\lambda$ defined as follows:
$$\begin{array}{rl} \t{\mc{ P}}_\lambda w:=&\partial_tw
-Tr \big(\t A(x,t)D^2 w\big) + \big(\t q(x,t) - 2 \lambda\t A(x,t) e \big) \cdot
Dw \\
&-\big( \lambda^2 e\t A(x,t) e +\lambda\t q(x,t) \cdot e +\t\mu(x,t)-\e \big)
w.\\
\end{array}$$
This will be achieved by showing that
\Fi{limitoperator}
\forall\lambda>0,\quad \t\kappa(\lambda)\leq\kappa(\lambda)+\e,
\Ff
whence $\t\kappa(0)<0$ as soon as $\e<-\kappa(0)$ (recall that $\kappa(0)<0$
by \eq{mu>0}).
Let us postpone for a moment the proof of \eq{limitoperator}.
Applying Proposition \ref{pro:nex} to $\t u$ yields
$$\liminf_{t\to+\infty}\frac1t\int_0^t\t c(s)ds\geq 
-\max_{\lambda>0}\frac{\t\kappa(\lambda)}\lambda=-\frac{\t\kappa(\hat\lambda)}
{\hat\lambda},$$
for some $\hat\lambda>0$.
In virtue of \eq{omegac} and \eq{limitoperator}, from this inequality we
deduce
$$\lm{c}\geq-\frac{\kappa(\hat\lambda)+\e}{\hat\lambda},$$
from which $\lm{c}\geq c^*$ follows by the arbitrariness of $\e$.

It remains to prove \eq{limitoperator}. Let $k>\kappa(\lambda)$. By 
definition \eq{phi=1} there exists $\vp$ such that $\inf \vp>0$, 
$\vp, D\vp \in L^\infty (\R^N\times \R)$ and $\mathcal{P}_\lambda \vp \leq k 
\vp$ in $\R^N\times \R$. We would like to perform on $\vp$ the same limit of 
translations as done before to obtain $\t u$ from $u$.
This would yield a function $\t\vp$ satisfying $\t{\mathcal{P}}_\lambda\t\vp 
\leq(k+\e)\t\vp$. But this argument requires the $L^p_{loc}$ estimates of the 
derivatives $\partial_t$, $D$, $D^2$ of the translated of $\vp$, which are not 
available because $\vp$ is a subsolution and not a solution of an equation.
However, it is possible to replace $\vp$ with a solution of a semilinear 
equation of the type $\mathcal{P}_\lambda w=g(w)$ in $\R^N\times \R$,
with $g$ smooth and such that $g(w)\leq (k+\e)w$, which satisfies the same 
properties as $\vp$, as well as the desired additional regularity properties.
This is done in the proof of Theorem A.1 of 
\cite{RossiRyzhik}, whose arguments can be exactly repeated here.
We can therefore apply the 
translation argument that provides a function $\t\vp$ such that 
$\t{\mathcal{P}}_\lambda\t\vp \leq(k+\e)\t\vp$. Moreover, $\inf\t\vp>0$ and 
$\sup\t\vp<\infty$. In order to be able to use $\t\vp$ in the definition of 
$\t\kappa(\lambda)$ and derive $\t\kappa(\lambda)\leq k+\e$, we only need to 
have that $\sup|D\t\vp|<\infty$. This property does not follow automatically 
from the $L^p$ estimates and the embedding theorem as in the elliptic case 
treated in 
\cite{RossiRyzhik}. This is the reason why we need the extra assumption 
\eq{hyp-Aextra} on $A$.
Indeed, we use Theorem 1.4 of \cite{PP}, with, using the same notations as in \cite{PP}, $F$ the nonlinear operator associated with equation $\mathcal{P}_\lambda w=g(w)$, 
Hypothesis 1.2 of \cite{PP} being satisfied since $A$ satisfies \eq{hyp-Aextra}, $q$ is bounded and $f=f(x,t,u)$ is bounded with respect to $(x,t,u)\in \R^N\times \R\times [0,1]$, and Hypothesis 
1.3 being satisfied with $\varphi (x,t):= e^{Mt} (1+|x|^2)$ and $M$ large enough. Hence, we get 
a uniform $L^\infty$ bound on $Dw$, where $w$ is the solution of  $\mathcal{P}_\lambda w=g(w)$. Using $w$ instead of $\vp$, we get that this bound is inherited by $\t\vp$ 
and we therefore deduce $\t\kappa(\lambda)\leq k+\e$.
%
As $k>\kappa (\lambda)$ is arbitrary, we eventually get \eq{limitoperator}.

\end{proof}



\end{document}